\documentclass[10pt,a4paper]{article}
\usepackage[margin=1.95cm]{geometry}
\usepackage{graphicx}
\usepackage{amsmath,amssymb}

\usepackage{amsthm}
\usepackage{latexsym}
\usepackage{bm}
\usepackage{verbatim}
\usepackage{graphicx}
\usepackage[caption=false]{subfig}
\usepackage{placeins}
\usepackage{dsfont}
\usepackage{microtype}
\DisableLigatures[f]{encoding = *}
\usepackage{ellipsis} 
\usepackage{algpseudocode}
\usepackage[ruled,linesnumbered]{algorithm2e}
\usepackage{xcolor}
\algnewcommand\And{\textbf{and}}
\usepackage[pdfstartview=FitH,pdfpagemode=None,colorlinks=true]{hyperref}
\usepackage{enumerate}
\usepackage{array}
\newcolumntype{C}[1]{>{\centering\arraybackslash}m{#1}}

\newcommand{\E}{\mathbb{E}}
\newcommand{\NN}{\mathbb{Z}_{>0}}
\newcommand{\RR}{\mathbb{R}}
\renewcommand{\SS}{\mathbb{S}^{d-1}}
\newcommand{\cK}{\mathcal{K}}
\newcommand{\PP}{\mathbb{P}}
\newcommand{\MH}{\mathrm{MH}}

\newcommand{\gq}{q_\cK}
\newcommand{\gP}{P_\cK}

\newcommand{\K}{A}
\newcommand{\Kc}{A^c}
\newcommand{\zv}{Z}
\newcommand{\rv}{R}

\newcommand{\cA}{\mathcal{A}}

\newtheorem{assumption}{Assumption}
\newtheorem{theorem}{Theorem}[section]

\newtheorem{proposition}[theorem]{Proposition}
\newtheorem{lemma}[theorem]{Lemma}

\begin{document}

\title{A Metropolis-class sampler for targets with non-convex support}

\author{John Moriarty\footnote{School of Mathematics, Queen Mary University of London, London E1 4NS. Email:\texttt{j.moriarty@qmul.ac.uk}}
\and
Jure Vogrinc\footnote{Department of Statistics, University of Warwick, Coventry, CV4 7AL. Email: \texttt{jure.vogrinc@warwick.ac.uk}}
\and
Alessandro Zocca\footnote{Department of Mathematics, Vrije Universiteit Amsterdam, 1081 HV Amsterdam. Email: \texttt{a.zocca@vu.nl}}}

\date{\today}

\maketitle
 
\begin{abstract}
We aim to improve upon the exploration of the general-purpose random walk Metropolis algorithm when the target has non-convex support $\K \subset \mathbb{R}^d$, by reusing proposals in $A^c$ which would otherwise be rejected. The algorithm is Metropolis-class and under standard conditions the chain satisfies a strong law of large numbers and central limit theorem. Theoretical and numerical evidence of improved performance relative to random walk Metropolis are provided. Issues of implementation are discussed and numerical examples, including applications to global optimisation and rare event sampling, are presented.\\


\end{abstract}

\section{Introduction}\label{sec:intro}

A key challenge for Markov chain Monte Carlo (MCMC) algorithms is the balance between global ``exploration'' and local ``exploitation''. In this paper we present the \textit{skipping sampler}, a general-purpose and easily implemented Metropolis-class algorithm which is capable of improving exploration of targets $\pi$ with nontrivial support $\K$, by reusing proposals lying outside $\K$. For this to be useful, we make the following standing assumption:
\begin{assumption}\label{ass:pi}
$\pi$ is a probability density function on $\RR^d$ whose support $$A= \textrm{supp}(\pi) := \{x \in \RR^d: \pi(x)>0\}$$ satisfies $Leb(A^c) > 0$, where $A^c$ is the complement of $A$ and $Leb$ denotes Lebesgue measure on $\RR^d$.
\end{assumption}
\noindent Such targets can arise for example in sampling from the superlevel sets of a density in the hybrid slice sampler \cite{neal2003}, or when sampling from rare events.

Proposals in $\K^c$ would be automatically rejected by standard algorithms such as random walk Metropolis (RWM), which exploits only local proposals for the next state of the chain. If a proposal lies in $\K^c$, the skipping sampler uses this information by attempting to cross $\K^c$ in a sequence of linear steps, much as a flat stone can jump repeatedly across the surface of water, and offer a relevant proposal. Since this can be seen as a tunnelling effect through the zero-mass region $\K^c$, it is advantageous when $\K$ is non-convex and, in particular, disconnected. The resulting Markov chain satisfies a strong law of large numbers and central limit theorem under essentially the same conditions as for RWM, to which we provide theoretical and numerical performance comparisons. 

To accelerate global exploration of the state space in MCMC algorithms, several approaches have by now been developed including tempering, Hamiltonian Monte Carlo and piecewise deterministic methods (see \cite{robert2018accelerating} for a recent review). However these methods are best suited to target densities with connected support, since the chain cannot cross regions where the target has zero density. A disconnected support would thus imply reducibility of the chain and its failure to converge to the target.

While RWM can be applied to targets with regions of zero density, its balance between exploration and exploitation can be problematic. If any state in $\K^c$ is proposed it is discarded and the chain does not progress. When $\K$ is non-convex, in particular, examples may be constructed where exploration is slow even when RWM is well tuned, making the chain sensitive to its initial state. This is illustrated in Figure \ref{fig:trajectories}, where red dots show the trace of a tuned RWM applied to a target with non-convex support, with four different initial states of the chain (the blue traces illustrate the increased exploration achieved by the skipping sampler).
One solution is to use knowledge of the target to design a more advanced proposal, such as those reviewed in \cite{robert2018accelerating}. However this approach is unavailable if the target density is unknown, or is known but insufficiently regular. In this case, a general-purpose method is instead required. 

\begin{figure*}[!ht]
\centering
\subfloat[ $X_0=(0,3.5)$]{\includegraphics[width=0.485\textwidth]{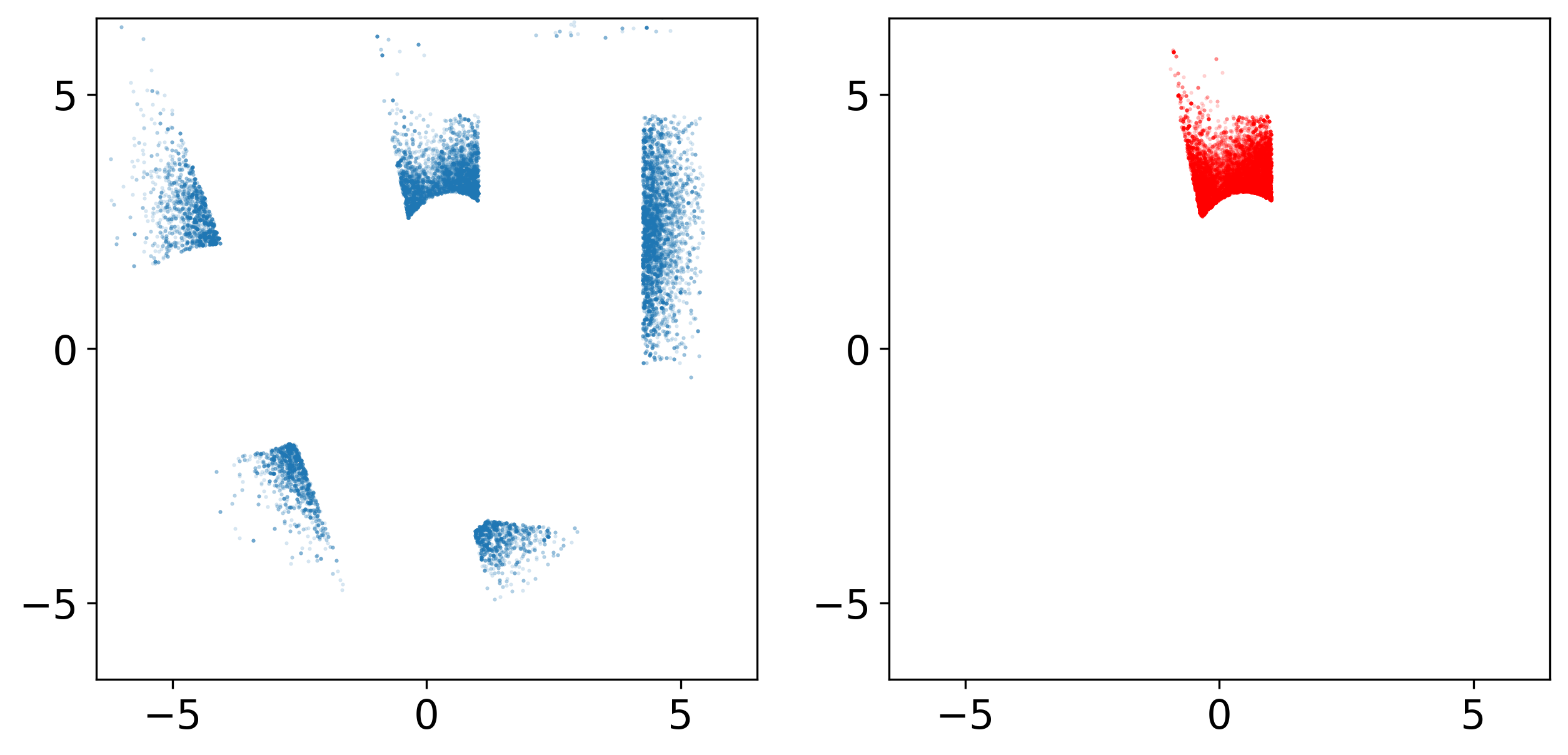}}
\hspace{0.22cm}
\subfloat[ $X_0=(2,-4)$]{\includegraphics[width=0.485\textwidth]{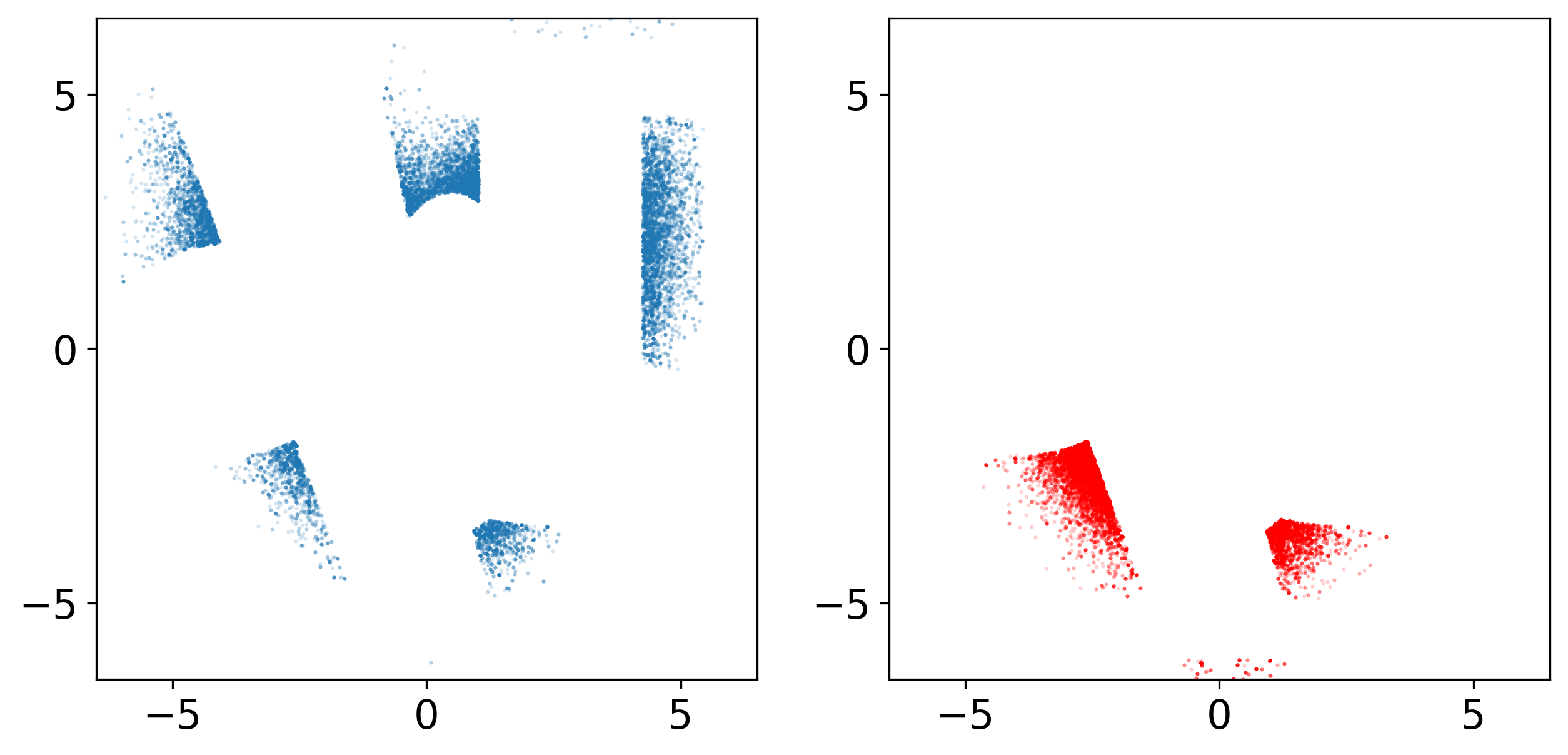}}
\\
\subfloat[ $X_0=(-4,-4)$]{\includegraphics[width=0.485\textwidth]{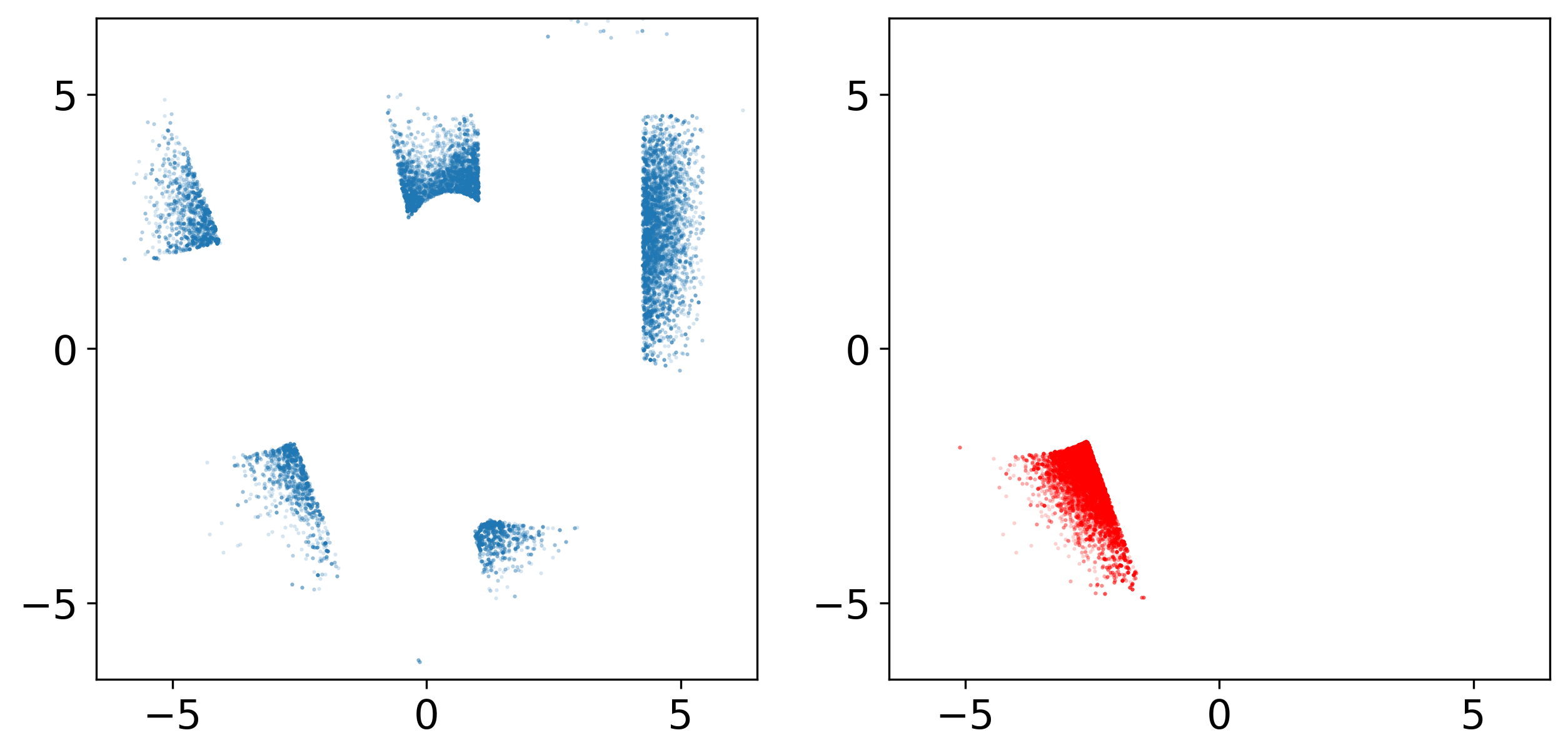}}
\hspace{0.22cm}
\subfloat[ $X_0=(-5,2)$]{\includegraphics[width=0.485\textwidth]{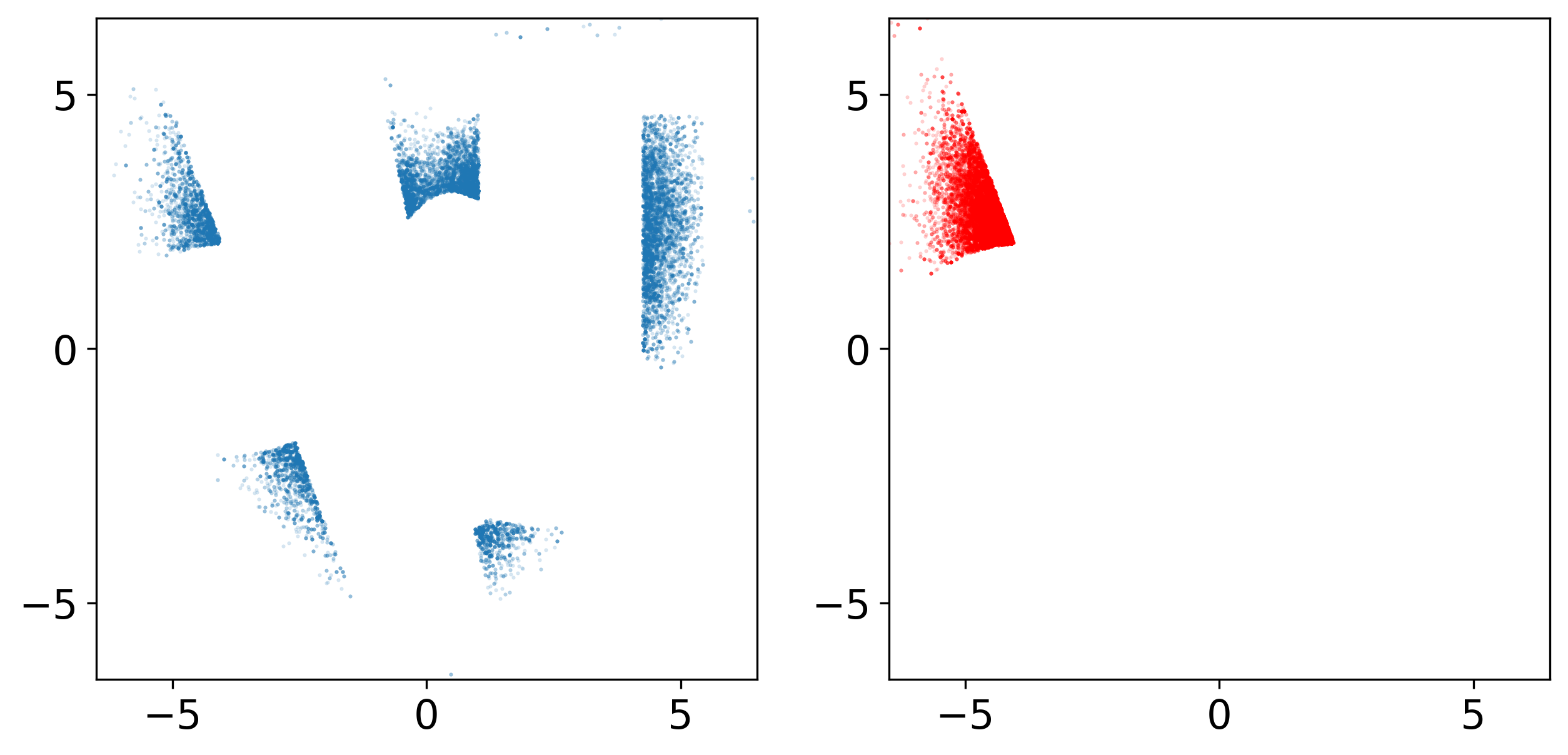}}
\caption{Traces of the proposed skipping sampler (blue) and RWM algorithm (red) when the target has disconnected support. Both samplers are started at the same initial point $X_0$ and use the same underlying Gaussian proposal, whose standard deviation is tuned for a RWM acceptance ratio of 25\%. The RWM typically localises around its initial state. 
}\label{fig:trajectories}
\end{figure*}

Theorem \ref{thm:CLT} establishes that the performance of the skipping sampler is at least as good as that of RWM according to the Peskun ordering. However the strengths of the proposed method lie primarily in applications to difficult low-dimensional problems. Conversely, in high dimensional problems the method generally offers similar performance to RWM. The aim of this paper is to present the method and illustrate its benefits via numerical examples, rather than to study any particular application exhaustively. 

Although it is not random walk-based, the skipping sampler is Metropolis class. The symmetry of the skipping proposal can be seen intuitively, provided that the direction of the first proposal is chosen symmetrically and the sequence of jump lengths has the same distribution when reversed. Thus although the proposal density typically does not have a convenient closed form, it need not be evaluated in order to access the Metropolis acceptance probabilities. Another advantage is that the sampler is general-purpose, in the sense that no knowledge is required of the target density beyond the ability to evaluate it pointwise. In particular, it is not necessary to know the target's support a priori.

Beyond the context of random sampling, our work has applications to probabilistic methods for deterministic non-convex optimisation such as multistart~\cite{Jain1993,Marti2003} and basin-hopping~\cite{Leary2000,Wales1997}. These methods combine deterministic local search, such as given by a gradient method, with random perturbations or re-initialisations which may be performed using the skipping sampler to improve exploration. Section \ref{sec:num} provides numerical examples of these applications.

\subsection{Related work}\label{sec:bgrw}
Many methods for accelerating the exploration of MCMC algorithms use prior knowledge of the target.
For example, known mode locations may be used to design global moves for the sampler \cite{andricioaei2001,Pompe2020,lan2014,sminchisescu2007,sminchisescu2003mode,tjelmeland2001}, or moves may be guided by the known derivatives of a differentiable target density \cite{lan2014,ram2018}.

Some exceptions are methods that generate multiple proposals, such as Multipoint MCMC \cite{qin2001} and Multiple-try Metropolis \cite{liu2000} which, like the skipping sampler, do not require additional information about the target. A fixed number of potentially correlated trial points are generated and one is selected at random, using a weight function which may be chosen to encourage exploration. Its random-grid implementation, in particular, has similarities with the skipping sampler. However, instead of fixing the number of draws, our proposal attempts to continue projecting further sequentially until it reaches $\K$. Another advantage of our method is that it is Metropolis class, which simplifies both implementation and theoretical analysis.

During the review process our attention was brought to the very interesting sequential proposals of \cite{Park2020}, which also introduces a Metropolis-class sampler that modifies the proposal sequentially. In the wider class of algorithms introduced there, it is possible to recognise methods close in spirit to the skipping sampler. When skipping is applied to the hybrid slice sampler as in Section~\ref{sec:slice}, for example, the resulting algorithm is a particular instance of the sequential proposal. While in~\cite{Park2020} the authors are motivated by the efficient implementation of Hamiltonian Monte Carlo, our own motivation is the efficient sampling of rare events. Together, these studies are  suggestive of further potential to use sequences of proposals to accelerate MCMC methods in a range of situations, for instance within the framework introduced in \cite{Andrieu20}.

Like the hit-and-run sampler \cite{Smith84} and related algorithms (see Section 6.3 of \cite{GilksRoberts1996}), the skipping sampler splits a Markovian transition into the random generation of a direction followed by a move in that direction. When the target conditioned on any line in the space is available in closed form, the hit-and-run algorithm is of course preferable, for the reasons provided in \cite{RudolfUllrich2018}. Otherwise (which is more typical in applications), the skipping sampler offers a simple alternative and has the potential to increase exploration in the case of a non-convex support.

While also designed for targets with non-convex support, the \textit{ghost sampler} introduced by the present authors in~\cite{mvz2018} is not general-purpose since it uses knowledge of the geometry of the set $\K$, assuming it is polyhedral. 

The rest of the paper is structured as follows. We introduce the skipping sampler in Section~\ref{sec:gs} and state our main results in Section~\ref{sec:mainthms}. Implementation and extensions are discussed in Section \ref{sec:implementation}. Numerical applications to slice sampling and rare event sampling are given in Section~\ref{sec:num}, together with an application to global optimisation. Section~\ref{sec:proof} is devoted to the proof of the main results. 

\section{Skipping sampler}\label{sec:gs}
In this section we introduce the skipping sampler on $\RR^d$, which is a modification of the RWM algorithm \cite{metropolis}. It is Metropolis-class although, unlike RWM, does not perform a random walk.
\begin{assumption}\label{ass:qSRW}
Let $q: \RR^d\to \RR_+\cup \{0\}$ be a symmetric ($q(x)=q(-x)$) continuous probability density function with $q(\bm{0})>0$. We refer to $q$ as the \textit{underlying} proposal density.
\end{assumption}

Recall that given the state $X_n$ of the chain, the RWM proposes a state $Y_{n+1}$ sampled from the density $y \mapsto q(y - X_n)$ and accepts it as the next state $X_{n+1}$ with probability
\begin{equation}\label{eq:alpha}
	\alpha(X_n,Y_{n+1}) \, \, \,:= \, \,\,
	\begin{cases}
		\min \left \{ 1, \frac{\pi(Y_{n+1})}{\pi(X_n)} \right \}
		& \text{ if } \pi(X_n) 
		\neq 0,\\
		1 & \text{ otherwise,}
	\end{cases}
\end{equation}
else it is rejected by setting $X_{n+1}=X_{n}$. Here $\pi$ is the target density, although we do not take care to distinguish between $\pi$ and the corresponding distribution as it will not cause confusion. For convenience we use the common shorthand $\MH(\pi,q)$ (after the more general Metropolis-Hastings algorithm, see \cite{hastings}) to refer to the Metropolis-class algorithm with target $\pi$ and proposal $q$. 

Algorithm \ref{alg:skip} presents the skipping sampler, which aims to endow RWM with an improved ability to cross regions in which the target has zero density. Beginning with a RWM proposal $Y_{n+1}$, it continues jumping in a linear trajectory and accepts or rejects the first state of nonzero target density to be encountered. Thus any RWM proposal $Y_{n+1} \in \Kc$, which would be rejected by $\MH(\pi,q)$, is instead reused by adding jumps of random size in the direction $Y_{n+1}-X_n$ until either $\K$ is entered, or skipping is halted.

\begin{algorithm}[!h]
\SetAlgoLined
\SetKwInOut{Input}{Input}\SetKwInOut{Output}{Output}
\Input{The $n$-th sample $X_n\in \RR^d$
}
\BlankLine
\DontPrintSemicolon
Set $X:=X_n$ and $Z_0=X$;\\
Generate the initial proposal $Y$ distributed according to the density $u \mapsto q(u-X)$;\\
Calculate the direction 
$$\Phi \quad=\quad \frac{(Y - X)}{\|Y - X\|}\,;$$\\
Generate an independent random halting index $K \sim \cK$;\\
Set $k=1$ and $Z_1:=Y$;\\
\While{$Z_k \in \Kc$ \textbf{\textup{and}} $k< K$}{
	Generate an independent distance increment $R$ distributed as $\|Y-X\|$ given $\Phi$;\\
	Set $Z_{k+1}=Z_k+ \Phi R$;\\
	Increase $k$ by one;\\
}
Set $Z:=Z_k$;\\
Evaluate the acceptance probability:
\begin{equation}\label{eq:ap2}
	\alpha(X,Z)\quad=\quad
	\begin{cases}
	\min\left(1, \frac{\pi(Z)}{\pi(X)} \right) & \text{ if } \pi(X) \neq 0, \\
	1, & \text{ otherwise; }
	\end{cases}
\end{equation}
Generate a uniform random variable $U$ on $(0,1)$;\\
\uIf{$U \leq \alpha(X,Z)$}{$X_{n+1}=Z$;}
\Else{$X_{n+1}=X$;}
\KwRet{$X_{n+1}$}.\;
\caption{Skipping sampler ($n$-th iteration)}\label{alg:skip}
\end{algorithm}

Algorithm \ref{alg:skip} can be interpreted as follows. The {\it halting index} $K$ is an independent random variable with distribution $\cK$ on $\NN\cup\{\infty\}$. If $K=1$ then $Y$, the usual RWM proposal, is taken as the proposal. However if $K>1$, the proposal is constructed using the \textit{skipping chain} $\{\zv_k\}_{k \geq 0}$ on $\RR^d$ defined by  $\zv_0:=X$, with $X=X_n$ being the current state of the chain, and the update rule%
\begin{equation}\label{eq:skippingchain}
	\zv_{k+1}:=\zv_k + \Phi \rv_{k+1}, \quad k \geq 0\,,
\end{equation}
where $\|\cdot\|$ denotes the Euclidean norm, $\Phi = (Y-X)/\|Y-X\|$, $R_1 =\|Y-X\|$, and the distance increments $\{\rv_k\}_{k \geq 2}$ are independent draws from the distribution of the radial part $\|Y-X\|$ conditional on the angular part $\Phi$.

\noindent Let $T_{\K}$ be the first entry time of the skipping chain into $\K$:
\begin{align}\label{eq:st}
	T_{\K}:=\min\{k \geq 1 ~:~ \zv_k  \in \K \},
\end{align}
with $\min \emptyset := \infty$. Writing $T_\K \wedge K$ for the smaller of the two indices $T_{\K}$ and $K$, we also require:
\begin{assumption}\label{ass:FiniteSkipping}
The support $\K=\mathrm{supp}(\pi)$ and distribution $\cK$ are such that 
$\E[T_{\K}\wedge K]~<~\infty\,.$
\end{assumption}

Relevant considerations for the choice of $\mathcal{K}$ and $q$ are discussed in Section \ref{sec:hi}. Note that almost surely we have both $Y \neq x$ (since $q$ is a density) and $T_{\K}\wedge K < \infty$ (Assumption~\ref{ass:FiniteSkipping}), so the skipping proposal $Z:=Z_{T_{\K}\wedge K}$ output by Algorithm \ref{alg:skip} is well defined.

\begin{proposition}\label{prop:m}
The following statements hold:
\begin{enumerate}[(i)]
\item Algorithm \ref{alg:skip} is a symmetric Metropolis-class algorithm on the domain $\K$. That is, there exists a transition density $q_\cK$ (which depends on the halting index distribution $\cK$) 
satisfying $q_\cK(x,z)=q_\cK(z,x)$ for all $x,z\in\K$, such that Algorithm \ref{alg:skip} is MH($\pi,q_\cK$).
\item The inequality $q_\mathcal{K}(x,z) \geq q(z-x)$ holds for every $x,z \in \K$.
\end{enumerate}
\end{proposition}

\begin{proof}
(i) We now make rigorous the intuitive argument which was provided earlier for the symmetry of the skipping proposal. Conditional on the direction $\Phi$, the skipping chain~\eqref{eq:skippingchain} is one-dimensional. We therefore analyse this one-dimensional chain, before integrating over $\Phi$ to obtain the unconditional transition density.

Consider transitions of the skipping chain~\eqref{eq:skippingchain} between the states $x$ and $z$ in exactly $k \in \NN$ steps. The intermediate states $ z_1, \ldots, z_{k-1}$ satisfy $ z_i \in \K^c$ for $i=1,\ldots,k-1$. The (sub-Markovian) density $z \mapsto \xi_k(x,z)$ of these transitions is given by the Chapman-Kolmogorov equation and the density $\xi(r)$ of the distance increment $R$, which can in $d$-dimensional spherical coordinates be seen to be proportional to $q(r\Phi)r^{d-1}$.
Since the distance increments are i.i.d.~and have symmetric densities ($q(-r\Phi)=q(r\Phi)$), simple manipulations of the Chapman-Kolmogorov integral confirm that it is unchanged when the start and end point, the order of the jumps, and the direction of each jump are all reversed. This establishes that the density $\xi_k$ is symmetric.

Next note that Assumption~\ref{ass:FiniteSkipping} implies the decomposition
\begin{align*}
 \{Z=z\} \quad  &=\quad
\bigcup_{k=1}^\infty\{Z=z,\; T_\K\leq k,\; K= k\} \, \cup~\{Z=z,\; T_\K < \infty,\; K= \infty\} \, \cup~\bigcup_{k=1}^\infty\{Z=z, \; T_\K > k,\;K=k\}\,.
\end{align*}
Hence, $Z$ given $x$ and $\Phi$ has a (sub-Markovian) density
\begin{align*}
\xi_\cK(x,z)\quad = &\quad
\sum_{k=1}^\infty \PP[K = k]\sum_{j\leq k}\xi_j(x,z)1_\K(z) \, +\, \PP[K= \infty]\sum_{j=1}^{\infty}\xi_j(x,z)1_\K(z) \,+\,\sum_{k=1}^\infty  \PP[K = k]\xi_k(x,z)1_{\K^c}(z)
\,.
\end{align*}
When $z\in\K$ the above can be simplified to
\begin{align}\label{eq:qcalk}
\xi_\cK(x,z)
\quad&=\quad
\sum_{k=1}^\infty \xi_k(x,z)\, \PP[K\geq k]\,.
\end{align}

Using $d$-dimensional spherical coordinates, the unconditional transition density is then the product of the density of $\Phi$ with the transition density conditional on $\Phi$:
\thinmuskip=2mu
\medmuskip=2mu
\thickmuskip=2mu
$$
	q_\cK(x,z) \, = \, \|z-x\|^{1-d}\xi_\cK(x,z)\cdot \int_{0}^{\infty}q\left(\frac{z-x}{\|z-x\|}r\right)r^{d-1}dr\,.
$$
\thinmuskip=3mu
\medmuskip=4mu
\thickmuskip=5mu
As the proposal $q$ and the densities $\xi_k$ (for all $k$) are symmetric, so is $\xi_\cK$ and so is the skipping proposal $q_\cK$, whenever $x,z\in \K$.

Since any proposal $Z \in \K^c$ is almost surely rejected if $x\in A$, Algorithm~\ref{alg:skip} is a well defined Metropolis-class algorithm on $A$, i.e. it is equivalent to MH($\pi,q_\cK$) on the domain $A$.

(ii) As noted above, if $K=1$ then Algorithm \ref{alg:skip} reduces to MH($\pi,q$). From \eqref{eq:qcalk} we therefore have
$$
\xi_\cK(x,z) \quad = \quad \sum_{k=1}^\infty \xi_k(x,z)\, \PP[K\geq k] \geq \xi_1(x,z) \cdot 1
$$
which again translates to the desired statement about proposal densities.
\end{proof}

\section{Theoretical results}
\label{sec:mainthms}

For completeness of the discussion below we provide the following definitions, further details of which may be found in \cite{tweedie}. A Markov chain $X_0,X_1\dots$ is $\pi$\textit{-irreducible} if for every $x\in \RR^d$ and every $D \subset \RR^d$ with $\pi(D)>0$ we have 
$$
	\PP_x\left[ \bigcup_{n \in \NN} \{ X_n\in D\} \right]\quad>\quad 0\,.
$$
Further, if $\PP_x\left[ \bigcup_{n \in \NN} \{ X_n\in D\} \right]=1$ for every $x\in B$ and every $D\subset B$ with $\pi(D)>0$ we say that $X_0,X_1,\dots$ is \textit{Harris recurrent on $B$}.
A set $B$ is \textit{absorbing} for a Markov chain with transition kernel $P$ if $P(x,B)=1$ holds for all $x\in B$. Note that an absorbing set $B$ gives rise to a Markov chain evolving on $B$ whose transition kernel is simply $P$ restricted to $B$ (see \cite[Theorem~4.2.4]{tweedie}).

It is clear from \eqref{eq:alpha} that if $x \in \mathrm{supp}(\pi)$ then
$$
    P(x,\mathrm{supp}(\pi)^c)=0,
$$
so that $\mathrm{supp}(\pi)$ is an absorbing set for the Metropolis algorithm with target $\pi$, and is a natural space of realisations of the chain.
In what follows we therefore always consider the chain to evolve on the set $\K$.

Regarding initialisation of the skipping sampler, note from~\eqref{eq:ap2} that if $X_0 \notin \textrm{supp}(\pi)$ in Algorithm~\ref{alg:skip} then $Z$ is automatically accepted. In this case the skipping sampler first enters $\textrm{supp}(\pi)$ at a random step $N$ and, for $0 \leq n \leq N-2$, we have $X_{n+1} = Z_K$ -- that is, the maximum allowed number of skips is performed at each stage. This implies that the skipping procedure is also capable of improving exploration in this initialisation stage. Theorem \ref{thm:ergodic} assumes that $\pi(X_0)>0$, or that initialisation has already been performed. 
We have

\begin{theorem}[SLLN]\label{thm:ergodic}
Suppose that $\MH(\pi,q)$ restricted to $\mathrm{supp}(\pi)$ is $\pi$-irreducible. Then $\MH(\pi,\gq)$ restricted to $\K=\mathrm{supp}(\pi)$ is also $\pi$-irreducible and Harris recurrent. Moreover, the Strong Law of Large Numbers holds: 
if $\{X_i\}_{i\in\NN}$ is the skipping sampler (generated by Algorithm~\ref{alg:skip}) initiated at $X_0=x\in A$, then for every $\pi$-integrable function $f$ we have
$$
	\lim_{n\to\infty}\frac{1}{n}\sum_{i=0}^n f(X_i)\quad \stackrel{a.s.}{=} \quad\int_{\RR^d}f(x)\pi(x)dx\,.
$$
\end{theorem}

The conditions of Theorem \ref{thm:ergodic}, which are mild, are discussed in Section \ref{sec:hi}. There are also cases where $\MH(\pi,q)$ is not irreducible but $\MH(\pi,\gq)$ is, for instance when the dimension $d=1$, $q$ is a random walk proposal with finite support, and $\K^c$ is an interval too wide to be crossed by a single random walk step, but which can be skipped across.

The statement of the second main result uses some additional notation (for further details see \cite{RobertsRosenthal97}). Consider the Hilbert space $L^2(\pi)$ of square-integrable functions with respect to $\pi$, equipped with the inner product (for $f,g\in L^2(\pi)$)
$$
	\langle f,g\rangle \, := \, \int_{\RR^d}f(x)g(x)\pi(x)dx \,=\, \int_\K f(x)g(x)\pi(x)dx.
$$
Since all Metropolis-class chains are time reversible, the Markov kernel of $\MH(\pi,q)$ defines a bounded self-adjoint linear operator $P$ on $L^2(\pi)$, defined for $f\in L^2(\pi)$ via
\begin{align*}
	Pf(x) \,  := & \, \int_{\RR^d}f(y)\alpha(x,y)q(y-x)dy + \left(1-\int_{\RR^d}\alpha(x,y)q(y-x)dy\right)f(x).
\end{align*}
If $P$ is irreducible then its operator norm is $\|P\|=1$, with $f \equiv 1$ as the unique eigenfunction for the eigenvalue $1$, and the spectral gap of $P$ is defined to be $\lambda:=1-\sup_{\{f~:~\|f\|=1,~\pi(f)=0\}}\langle Pf,f\rangle$.

\begin{theorem}
\label{thm:CLT}
Under the conditions of Theorem~\ref{thm:ergodic}, denoting respectively by $P$ and $\gP$ the Markov kernels of $\MH(\pi,q)$ and $\MH(\pi,\gq)$ restricted to $\K=\mathrm{supp}(\pi)$, the following statements hold:
\begin{enumerate}[i)]
	\item For every $f\in L^2(\pi)$ we have $\langle \gP f,f\rangle\leq \langle P f,f\rangle$;
	\item If $\MH(\pi,q)$ has a non-zero spectral gap $\lambda$, then $\MH(\pi,\gq)$ also has a non-zero spectral gap $\lambda_\mathcal{K}$ that satisfies $\lambda_\mathcal{K}\geq \lambda$;
	\item If the central limit theorem (CLT) holds for $\MH(\pi,q)$ and function $f$ with asymptotic variance $\sigma^2(f)$, that is
	$$
		\sqrt{n}\left(\frac{1}{n}\sum_{i=0}^nf(X_i)-\pi(f)\right) \quad
		\longrightarrow\quad N(0,\sigma^2(f))
	$$
	in distribution, then the CLT also holds for $\MH(\pi,\gq)$ and the same function $f$, with asymptotic variance $\sigma_\cK^2(f)$ satisfying $\sigma_\cK^2(f) \leq \sigma^2(f)$. 
\end{enumerate}
\end{theorem}

The inequality at point (i) of Theorem~\ref{thm:CLT} gives a useful way to compare performance and mixing of different Markov kernels. Indeed, one can consider the Peskun-Tierney partial ordering (see \cite{Peskun73} and \cite{mira2001ordering,mira2009covariance,Tierney1998}) on the family of bounded self-adjoint linear operators on $L^2(\pi)$ by setting $P_1\geq P_2$ whenever $\langle P_1 f,f\rangle\leq \langle P_2 f,f\rangle $ holds for all $f\in L^2(\pi)$. 

Intuitively, point (ii) of Theorem~\ref{thm:CLT} means that the skipping sampler has the potential to mix faster than the classical random walk Metropolis, i.e., converge to stationarity in fewer steps. As explained in Section 2.1 of \cite{RudolfUllrich2018} the speed of convergence to stationarity can also be measured by other analytical quantities of the form $\inf_{f\in \mathcal{M}}\langle(I-P)f,f\rangle$ for some subset $\mathcal{M}$ of $L^2(\pi)$; it is straightforward to modify Theorem~\ref{thm:CLT} accordingly. In the case of the spectral gap presented above we have $\mathcal{M}=\{f\in L^2(\pi) ~:~ \pi(f)=0\text{ and }\pi(f^2)=1\}$. It follows from point (iii) that in stationarity, the samples produced by the skipping sampler are at least as good for estimating $\pi(f)$ as those generated by RWM.

These theoretical benefits are balanced by increased computational complexity. The exploration added by the skipping sampler relative to RWM carries a computational cost, and the tradeoff between cost and benefit depends on the target density. In particular, this tradeoff could become disadvantageous if evaluating the target density (and thus assessing the event $\{ Z_k \in A \}$) in Algorithm \ref{alg:skip} carries high cost. In the absence of global knowledge of the target, a pragmatic approach would be to run both methods and try to judge between their output. In Section~\ref{sec:slice}, for example, we have compared the mean squared error of the coordinate projections against the increased number of evaluations of the target density. 
As noted in Section \ref{sec:q}, the evaluations of the target density can also be vectorised with the aim of decreasing computation time.

Sufficient conditions for parts (ii) and (iii) of Theorem~\ref{thm:CLT} have been studied in the literature. An aperiodic reversible Markov chain has non-zero spectral gap if and only if it is geometrically ergodic (see \cite{RobertsRosenthal97}), a property which is explored  in~\cite{jarner,mengersen,roberts2} for random walk Metropolis algorithms. The CLT holds essentially for all $f\in L^2(\pi)$ under the assumption of geometric ergodicity (see \cite[Section~5]{roberts}), but also holds more generally (see \cite{JarnerRoberts}).

\section{Implementation and extensions}\label{sec:implementation}
\label{sec:hi}
Implementing Algorithm \ref{alg:skip} involves two choices, an underlying proposal density $q$ and a halting index $\cK$, which are discussed respectively in Sections \ref{sec:q} and \ref{sec:K}. An alternative to Algorithm \ref{alg:skip} using a `doubling trick' for greater computational efficiency is given in Section \ref{sec:doubling}.

\subsection{Choice of $q$}\label{sec:q}

In addition to Assumptions \ref{ass:qSRW}--\ref{ass:FiniteSkipping}, to ensure that the SLLN holds (Theorem \ref{thm:ergodic}) we require $\MH(\pi,q)$ to be $\pi$-irreducible. This holds, for example, when $\pi$ is continuous and bounded and $q$ is everywhere positive. 
More generally, $\MH(\pi,q)$ is also irreducible if the interior of $\K$ is (non-empty) connected and there exist $\delta,\epsilon>0$ such that $q(x)>\epsilon>0$ whenever $\|x\|<\delta$ (see \cite[Section~2.3.2]{tierney}).

Since skipping can be seen as a way of endowing RWM with an improved ability to cross regions of zero density, a minimal approach would be to tune $q$ as if it were to be employed in the RWM algorithm, for example achieving an acceptance ratio around $25\%$ when $q$ is employed in RWM. However we have observed empirically that a lower acceptance ratio, for example 15\%, may further stimulate skipping.

\subsubsection{Computational aspects}

For sampling of the i.i.d.~radial increments $R_1,R_2,\dots$, it is desirable to choose $q$ such that samples may be drawn efficiently from
\begin{align}\label{eq:rcondphi}
\|Y-X\| ~\text{ conditional on }~ \Phi=\frac{Y-X}{\|Y-X\|}=\varphi,
\end{align}
for all $\varphi\in\SS$. Convenient cases include when $q$ is radially symmetric so that conditioning is not required, or when $q \sim \mathcal N(0,\Sigma)$ for some  $d\times d$ covariance matrix $\Sigma$, so that, given direction $\varphi$, each increment $R_i$ follows a generalised gamma distribution with density
$$
\frac{(\varphi^T \Sigma^{-1} \varphi)^{d/2}}{2^{d/2-1} \Gamma(\frac{d}{2})} \,r^{d-1}\, e^{-(\varphi^T \Sigma^{-1} \varphi) \frac{ r^2 }{2}  } \,.
$$
Alternatively one may specify $q$ indirectly by choosing the unconditional distribution of $\Phi$ and the conditional distribution of $R:=\|Y-X\|$ given $\Phi$, then checking that the conditions of Theorem \ref{thm:ergodic} are satisfied.

If sampling from the distribution \eqref{eq:rcondphi} is computationally expensive, however, the sampler may be modified by setting all $R_k$ equal to $R$, so that only a single sample is required to generate a proposal and the skipping chain keeps moving in the direction $\Phi$ with jumps of equal size. While this modification would not change the mean distance $\|Z_m-Z_0\| = \sum_{i=1}^m R_i$ covered by $m$ steps of the skipping chain, it would increase its variance to $m^2 \mathrm{Var}(R)$. 

\subsubsection{Anisotropy \label{sec:aniso}}
If $\K$ has a known anisotropy, the angular part of the underlying proposal may be chosen to favour certain directions in comparison to others, for example by tailoring the covariance matrix in a normal proposal $q\sim\mathcal{N}(0,\Sigma)$. This may be useful in high dimensional problems where otherwise, with high probability, the skipping chain may fail to re-enter $\K$.  It is not difficult to show that if the distance increment retains the properties used in the proof of Proposition \ref{prop:m}, then the acceptance ratio~\eqref{eq:alpha} depends additionally on the ratio of the angular densities. Denoting by $q_\varphi(x,\phi)$ the density of direction $\phi$ at the location $x$, for $\Phi=\frac{Y_{n+1}-X_n}{\|Y_{n+1}-X_n\|}$ the acceptance probability then equals
$$
\alpha(X_n,Y_{n+1})
\quad=\quad
\min\left(1,\frac{\pi(Y_{n+1})q_\varphi(Y_{n+1},-\Phi)}{\pi(X_n)q_\varphi(X_n,\Phi)}\right)\,.
$$

Although beyond the scope of this paper, in the absence of geometric knowledge of $A$ other information, for instance the history of the chain, may be used in an online fashion to make the angular part of the underlying proposal density dependent on the chain's current location. 

\subsection{Choice of $\cK$}\label{sec:K}

The simplest choice is a nonrandom halting index $K \equiv k_s \in \mathbb{Z}_{>1}$. Under this choice the $k_s$ skips can be vectorised and stored in memory along with the corresponding states $Z_i$ for $i=1,\dots k_s$, and the evaluations of whether  $Z_i\in\K$ for $i=1,\dots k_s$ can then be performed in parallel. This increases computational speed at the expense of a $k_s$-times higher memory requirement plus the coordination cost of parallelisation, and the balance between benefit and cost is not explored here. However if the additional computational costs are low, and if the costs of evaluating whether $Z_i\in\K$ are bounded, then the skipping sampler may be run at speed approaching that of RWM.

There is of course interplay between the choices for $\cK$ and $q$. For example, if an upper bound $D$ is available for the diameter of $\K^c$ then we may use $k_s=\frac{D}{\sup_{\varphi}\sigma_\varphi}$, where $\sigma_\varphi$ denotes the standard deviation of the conditional jump density in the direction $\varphi$. In the anisotropic case of Section \ref{sec:aniso}, {\it mutatis mutandis} the halting index may also be made direction-dependent using a parametric family of constants (or distributions) $\cK_\varphi$, $\varphi \in \SS$. To preserve symmetry it is then necessary that $\cK_\varphi = \cK_{-\varphi}$ for each $\varphi \in \SS$. Similar tradeoffs between $\cK$ and $q$ may also be made when $\cK$ is chosen to be random with finite mean.

If skipping cannot be efficiently parallelised as suggested above then, clearly, large realisations of $\cK$ can result in high computational costs if $\K$ is not re-entered. In the extreme, bearing in mind Assumption \ref{ass:FiniteSkipping}, an unbounded distribution $\cK$ should only be taken if $A^c$ is known to be bounded. If $\cK$ cannot be chosen based on a known diameter $D$ as above, then the absolute length of skipping trajectories may alternatively be controlled probabilistically using a large deviations estimate, as follows. If the conditional jump distribution is $R$ then the probability that a distance $m r$ can be traversed in $m$ skips is approximately (see for example \cite{Dembo2010}):
$$
	\PP \left (\sum_{k=1}^m R_i \geq m r \right ) \approx \mathrm{exp}(-m I(r)), 
$$
where $I(r)=\sup _{\theta >0}[\theta r-\lambda (\theta )]$ is the Legendre-Fenchel transform of $R$, provided that $R$ has finite logarithmic moment generating function, i.e. $ \lambda (\theta )=\ln \E [\exp(\theta R)] < \infty$ for all $\theta \in \RR$. 

Based on the above, if $\cK$ is random and mass is to be placed on large values of $\cK$ then this could lead to large computational costs. In this case the doubling trick of Section \ref{sec:doubling} may be applied.

\subsection{The doubling trick}\label{sec:doubling}

For clarity of exposition we first assume that $A^c$ is convex. From \eqref{eq:skippingchain}, the state $Z_k$ of the skipping chain is the partial sum $x + \Phi \sum_{i=1}^k R_i$, where the $R_i$ are i.i.d.~and $R_1=\|Y-x\|$. Recalling \eqref{eq:st}, define 
\begin{align}
	T_{\K}&:=\min\{k \geq 1 ~:~ \zv_{k}  \in \K \},\\
	S_{\K}&:=\min\{k \geq 1 ~:~ \zv_{2^k-1}  \in \K \}.
\end{align} 
The convexity of $A^c$ induces an ordering on the skipping chain, in the sense that 
\begin{align}
Z_k &\in A^c, \quad \text{ if } k < T_A, \\
Z_k & \in A, \quad \text{ if } k \geq T_A.
\label{eq:order}
\end{align}
If $T_A < K$ then Algorithm \ref{alg:skip} evaluates $T_A$ by sampling the partial sums $\{Z_k\}_{k \geq 1}$ sequentially. The following alternative implementation evaluates $T_\K$ significantly faster, in order $\log_2T_{\K}$ steps. It requires that for any $k$, the sum $\sum_{i=1}^k{R_i}$ may be sampled directly, both unconditionally and given the value of $\sum_{i=1}^{2k}{R_i}$, at a comparable cost to sampling $R_1$. This is possible, for example, when the $R_i$ are exponentially distributed.

The idea is to search forward through the exponential subsequence $Z_1,Z_3,Z_7\dots Z_{2^k-1},\dots$ until $k= \tilde k = S_{\K}$ (so that $Z_{2^{\tilde{k}}-1} \in A$), and then to perform a logarithmic search \cite{2008data} of the sequence $Z_{2^{\tilde{k}-1}-1}, \ldots, Z_{2^{\tilde{k}}-1}$ to identify $T_{\K}$. That is, sample $Z_m$ for $m=2^{\tilde{k}-1}-1+2^{\tilde{k}-2}$ and then, depending on whether or not it lies in $\K$, reduce the search to either the sequence $Z_{2^{\tilde{k}-1}-1}, \ldots, Z_{2^{\tilde{k}-1}-1+2^{\tilde{k}-2}}$ or the sequence $Z_{2^{\tilde{k}-1}-1+2^{\tilde{k}-2}}, \ldots, Z_{2^{\tilde{k}}-1}$, repeating until $T_{\K}$ is found.

For generalisations of this trick, note first that the doubling trick can be used only to accelerate skipping over a convex subset $B \subset A^c$, so that we only add a single distance increment at a time while the skipping chain is in $A^c \setminus B$, and use the doubling trick while in $B$. The idea may then be applied to a maximal convex subset of $A^c$, provided that such a subset is known. Then note that if $B_1,\ldots,B_{n_B}$ are all convex subsets of $A^c$, the doubling trick may be used to traverse each convex subset $B_i$ in turn, if needed. Thus the idea may be applied to an inner approximation of $A^c$ by a union of balls, for example.

\section{Numerical examples}\label{sec:num}

In order to motivate some applications, Section \ref{sec:0} begins with a general discussion of targets for which the skipping sampler offers an advantage over RWM. The numerical example of Section \ref{sec:example3}, in the context of rare event sampling, illustrates an improvement in exploration achieved by our method. Then, in an application to optimisation, Section \ref{sec:example5} provides quantitative examples of performance improvements obtained when the skipping sampler is used as a subroutine in probabilistic methods for non-convex optimisation. The Python code used to create all these numerical examples and figures is available at~\cite{Zocca2021}.

\subsection{General considerations}\label{sec:0}

Note firstly that if the initial proposal $Y$ lies in $A^c$ then it would be rejected by the RWM algorithm. Instead, in Algorithm \ref{alg:skip} it is reused. Thus skipping offers an advantage over RWM if the initial proposal $Y$ regularly lies in $A^c$.
Secondly, when $Y \in A^c$ the skipping proposal $Z$ of Algorithm \ref{alg:skip} needs regularly to be accepted (which in turn necessitates $Z \in A$). By construction (since $Z$ lies beyond $Y$ on the straight line between the current state $X_n \in A$ of the chain and $Y \in A^c$), this requires the support $A$ of the target to be non-convex.

The dimension $d$ also plays a key role. Considering an example where the support $A$ is the union of two disjoint balls in $\mathbb{R}^d$, by increasing $d$ we reduce the probability that $Z \in A$. 
Hence, the benefit of skipping is greatest in low dimensions and then gradually decreases. Nevertheless, in Section~\ref{sec:example3} we show that in special cases the sampler can be beneficial even in high dimensions. 

We also note the following tradeoff. Due to the increased exploration offered by the skipping sampler, the density encountered upon landing at $Z \in \K$ after crossing $\K^c$ may be significantly different from that at the current state $X_n$ of the Markov chain. In particular, if the target density does not vary slowly then the acceptance ratio $\alpha(X,Z)$ may be so low that such skips are not regularly accepted. Although this tradeoff is problem dependent, it does not apply in the rare event example of Section \ref{sec:example3}.

\subsection{Hybrid Slice Sampler}\label{sec:slice}

The slice sampler may be used to sample from a density $\rho$ on $\mathbb{R}^d$ as follows. Given the current sample $X_n \in K$, the following two steps generate the next sample $X_{n+1}$: 
\begin{enumerate}
\item[(i)] pick $t$ uniformly at random from the interval $ [0,\rho(X_n)]$,  
\item[(ii)] sample uniformly from the `slice' or superlevel set $$\K(t) := \{x \in K: \rho(x) \geq t\}.$$
\end{enumerate}
We refer the reader to \cite{Latuszynski2014,neal2003} and references therein for more information on the slice sampler and its convergence properties.

Step (ii) is typically infeasible in multidimensional settings. Instead, in the \textit{Hybrid Slice Sampler} (HSS) a Markov chain is used to approximately sample the uniform distribution on the slice. 
The following example illustrates the potential advantage of using the skipping sampler rather than RWM to generate this chain, since the slice may not be convex.

For $\rho$ we take a uniform mixture of $m=7$ standard normal densities in $d=5$ dimensions, whose means are drawn uniformly at random from a box $B=[-12,12]^5$. The underlying RWM proposal is a spherically symmetric Gaussian, with variance tuned to achieve an acceptance ratio of $23.5\%$ in RWM. Independent trajectories (started in stationarity) of $n=2\cdot 10^5$ steps were generated for the HSS algorithm with respectively the RWM and the skipping sampler used to sample from the superlevel sets. The halting index is taken to be $\PP[K_\varphi=15]=1$ for all $\varphi\in \SS$.

\begin{figure*}[!h]
\centering
\subfloat[Scatter plot of first two coordinates of HSS with RWM.]{\includegraphics[width=0.42\textwidth]{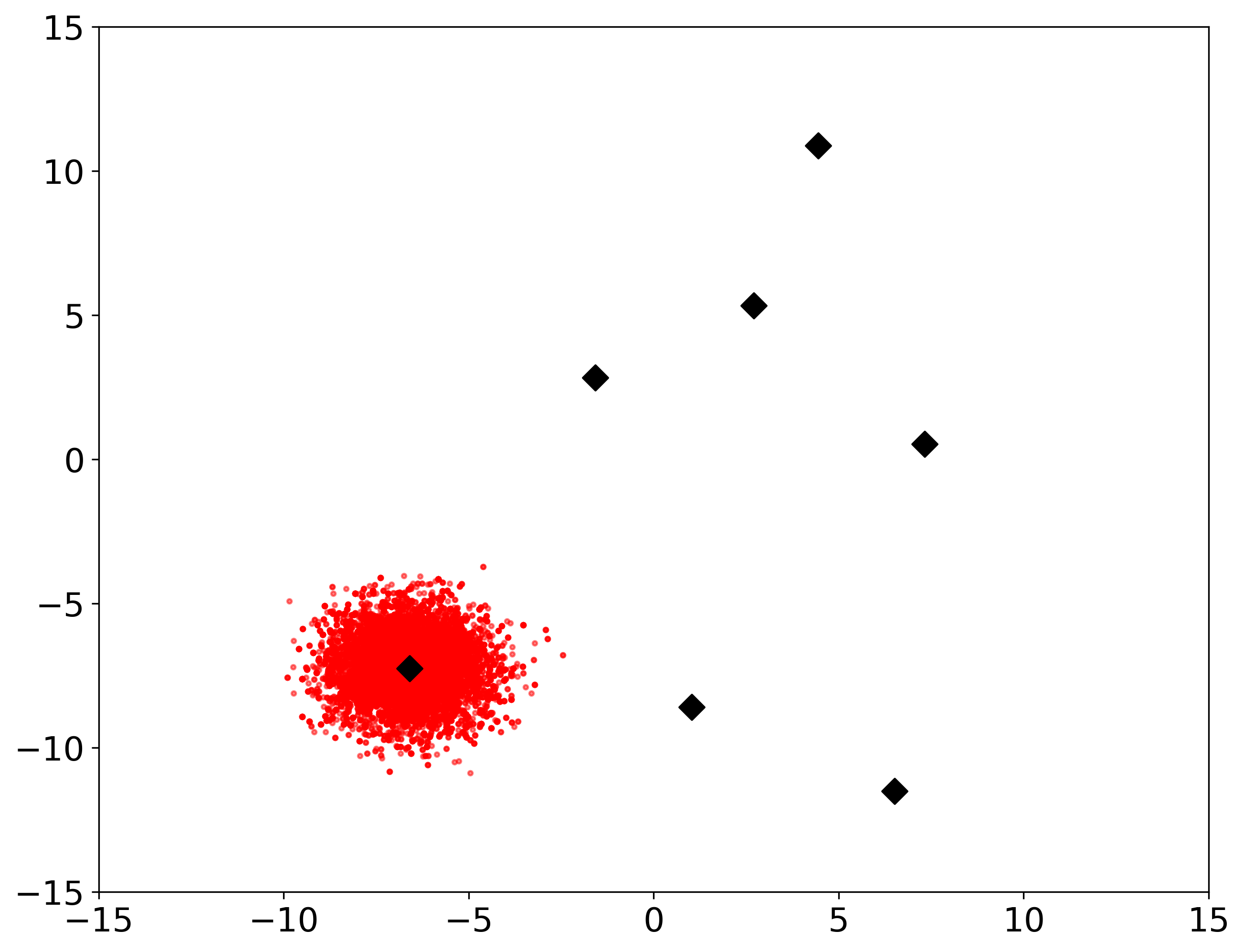}}
\hspace{0.8cm}
\subfloat[Scatter plot of first two coordinates of HSS with skipping sampler trajectory.]{\includegraphics[width=0.42\textwidth]{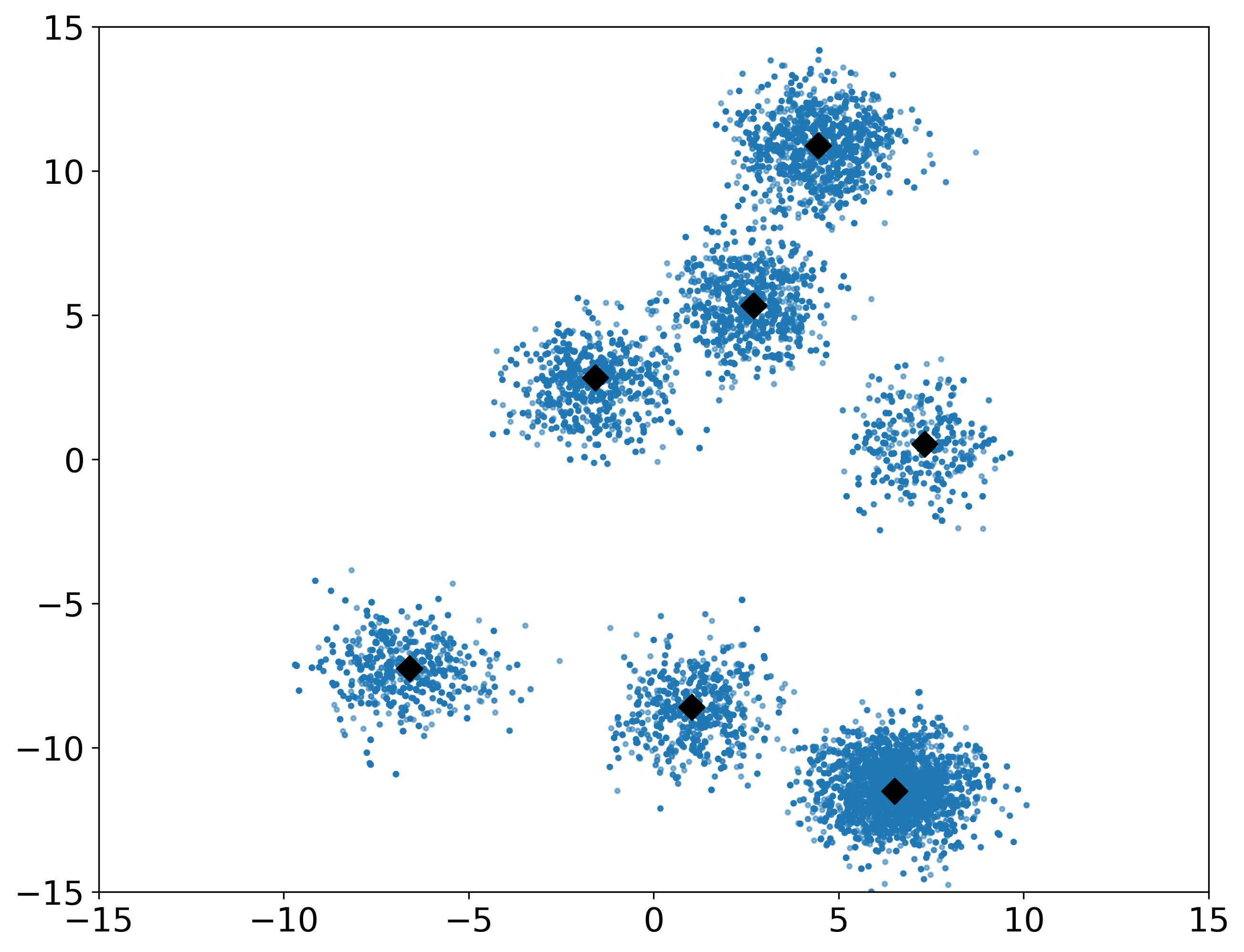}}
\\
\subfloat[First coordinate trajectory of HSS with RWM]{\includegraphics[width=0.435\textwidth]{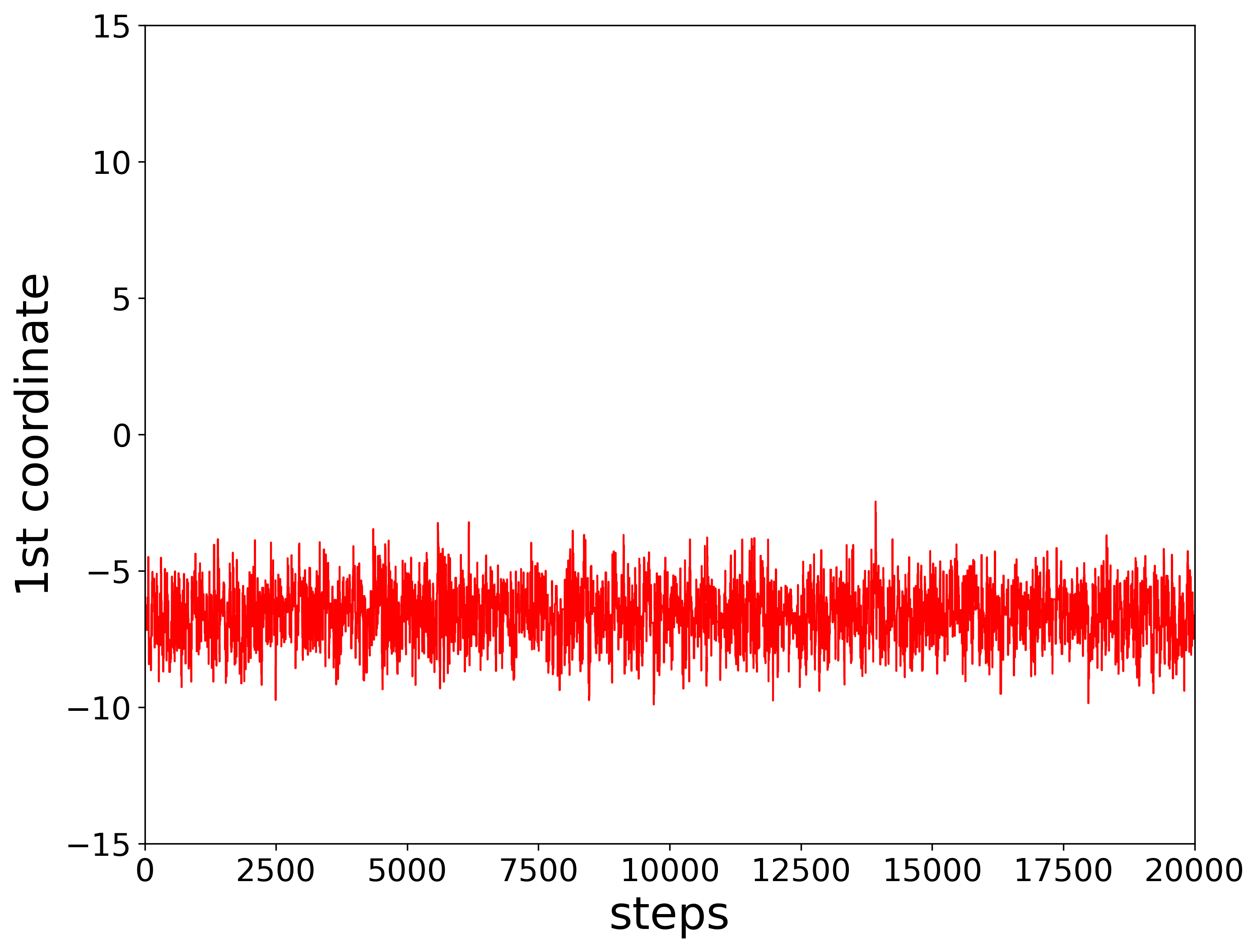}}
\hspace{0.6cm}
\subfloat[First coordinate trajectory of HSS with skipping sampler]{\includegraphics[width=0.435\textwidth]{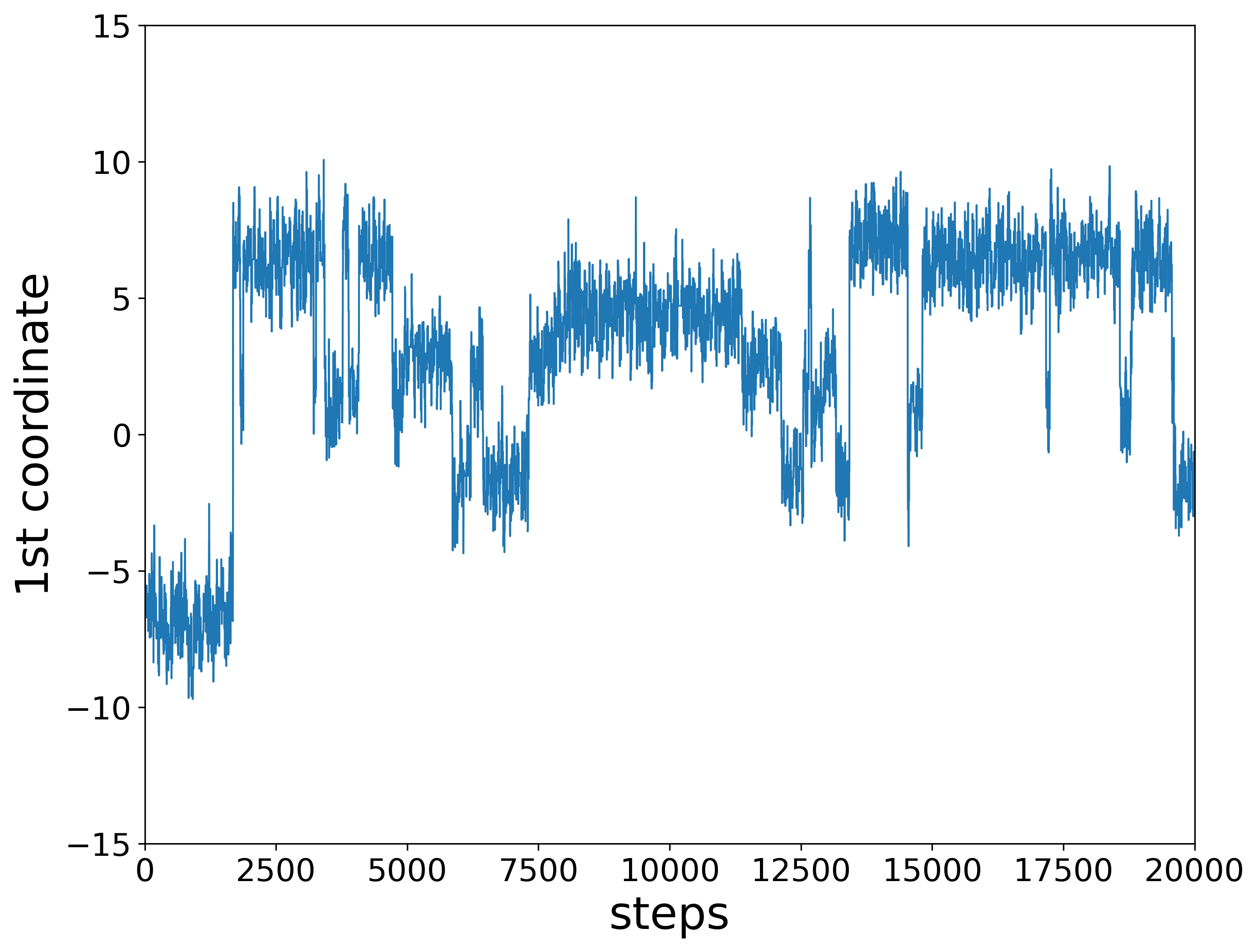}}
\caption{Comparison of RWM and skipping sampler as subroutines in HSS when started at the same point. Black diamonds in (a) and (b) represent true mode centers.}\label{fig:slice}
\end{figure*}

As can be seen from Figure~\ref{fig:slice}, the RWM implementation remains in the mode in which it was initiated. In contrast, the skipping sampler version transitions regularly between the seven modes. The experiment was run $m=100$ times (on the same Gaussian mixture), during which skipping transitions happened on average $16$ times per run. While the number of evaluations of the target density increased $11.61$ fold on average, skipping greatly reduced the mean squared error (MSE) for the estimators of the coordinate-wise means. The MSE for RWM and reduction estimates (MSE for RWM divided by MSE for skipping) for each of the five coordinates are reported in Table~\ref{tab:MSE}. Hence, in this example the skipping sampler is roughly $12$ times more expensive to compute, but produces samples with much greater effective sample size.
\begin{table}[h!]
	\begin{center}
	\begin{tabular}{c|ccccc}
	coordinate  & 1 & 2 & 3 & 4 & 5 \\ \hline \hline
	MSE RWM  &  22.85 & 58.45 & 48.09 & 35.65 & 44.20 \\ \hline
	MSE Skipping & 3.90 & 1.24 & 0.003 & 2.84 & 0.750 \\ \hline
	reduction & 5.86 & 47.23 & 18504 & 12.54 & 58.91
	\end{tabular}
	\end{center}
\caption{Mean squared errors for HSS with RWM or skipping sampler and its ratio}%
\label{tab:MSE}%
\end{table}%

\subsection{Rare event sampling}\label{sec:example3}

The aim in this example is to sample rare points under a complex density $\rho$ on $\RR^d$, by sampling from its intrinsic tail or sublevel set $\K=\{x\in\RR^d ~|~ \rho(x)\leq a\}$ for some $a>0$. 
As an illustration let $\rho$ be a mixture of $m=20$ Gaussian distributions, with randomly drawn means, covariances and mixture coefficients. 

We use the tails given by the levels $a=e^{-15}$ and $a=e^{-350}$ respectively for dimensions $d=2$ and $d=50$. In the case $d=2$, a visual illustration of Theorem \ref{thm:CLT} is provided by plotting comparisons of the exploration achieved in $10^5$ steps of RWM and the skipping sampler respectively. Since the superlevel sets of a finite Gaussian mixture are bounded, in this example we may take the halting index $K=\infty$.

\begin{figure*}[!ht]
\centering
\subfloat[Scatter plot of RWM trajectory]{\includegraphics[width=0.42\textwidth]{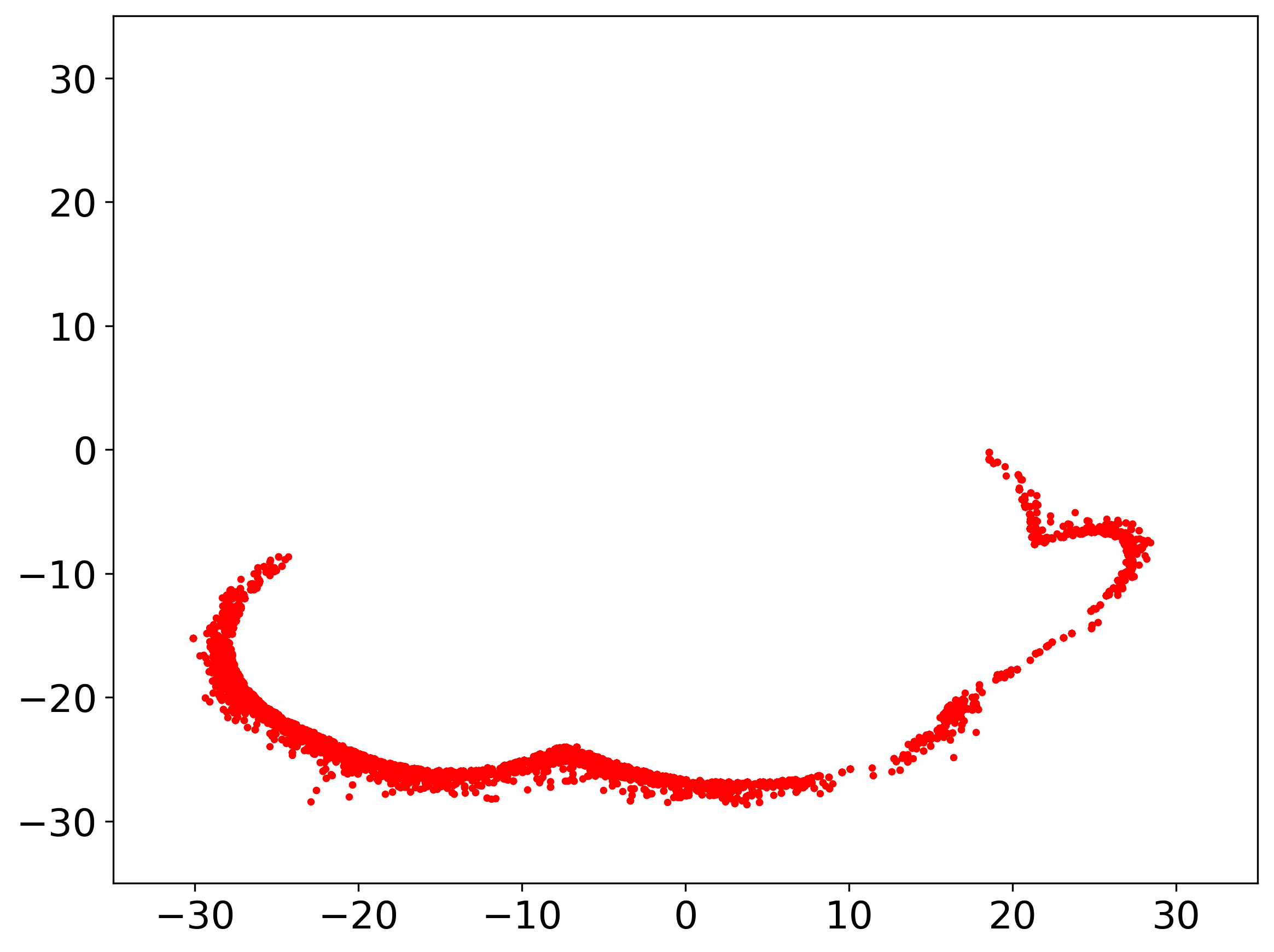}}
\hspace{0.85cm}
\subfloat[Scatter plot of skipping sampler trajectory]{\includegraphics[width=0.42\textwidth]{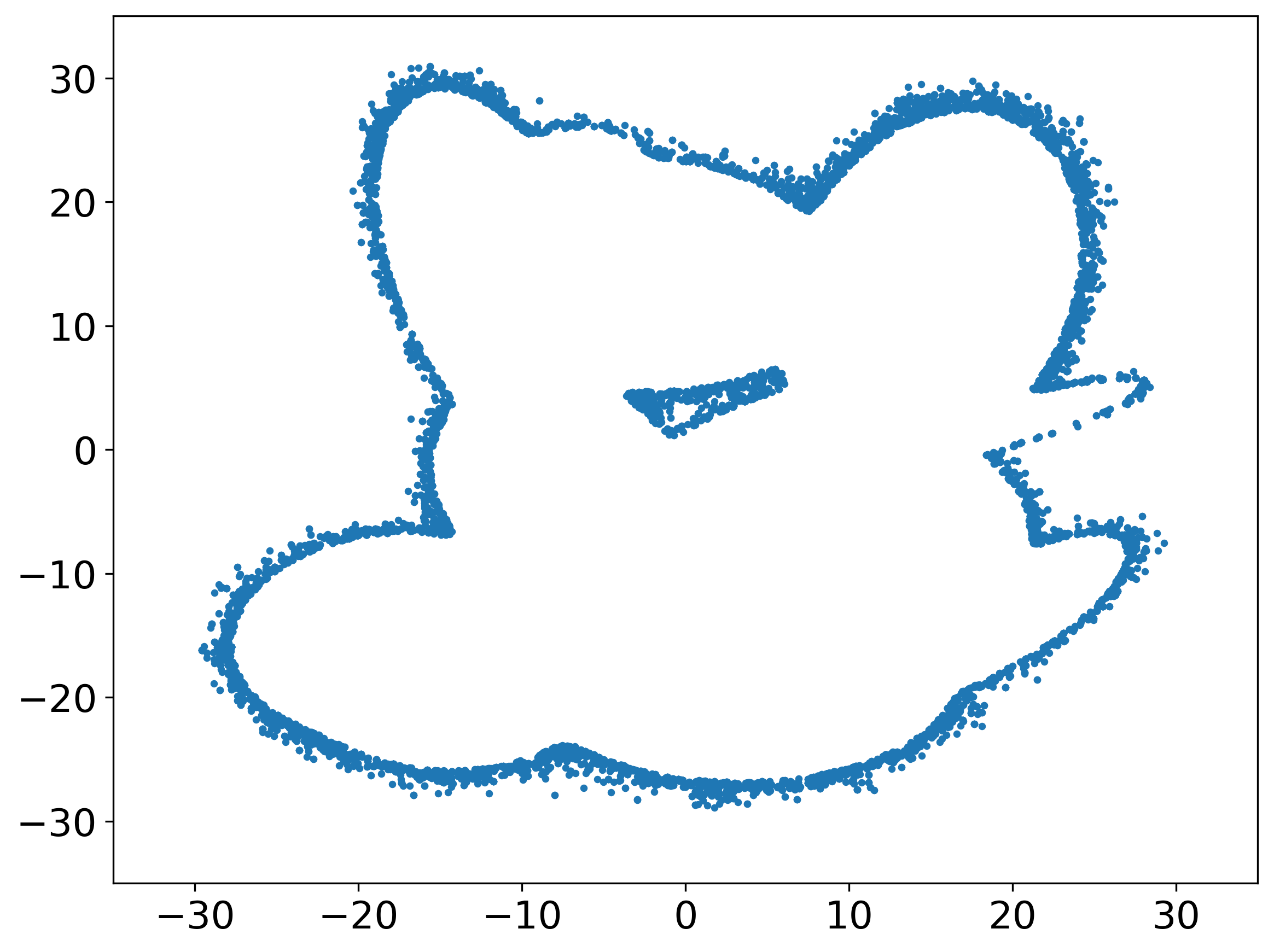}}
\\
\subfloat[First coordinate trajectory of RWM]{\includegraphics[width=0.435\textwidth]{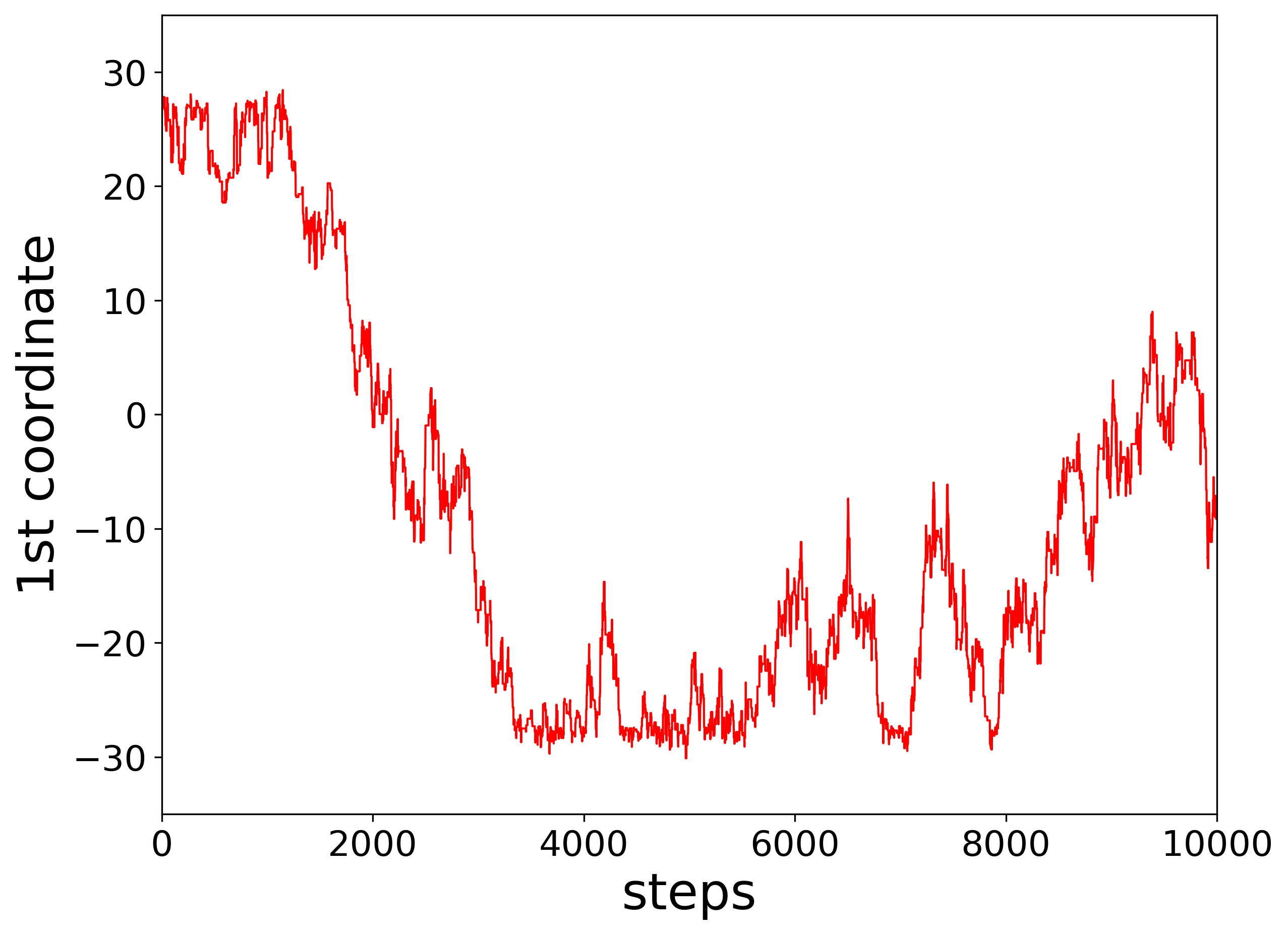}}
\hspace{0.6cm}
\subfloat[First coordinate trajectory of skipping sampler]{\includegraphics[width=0.435\textwidth]{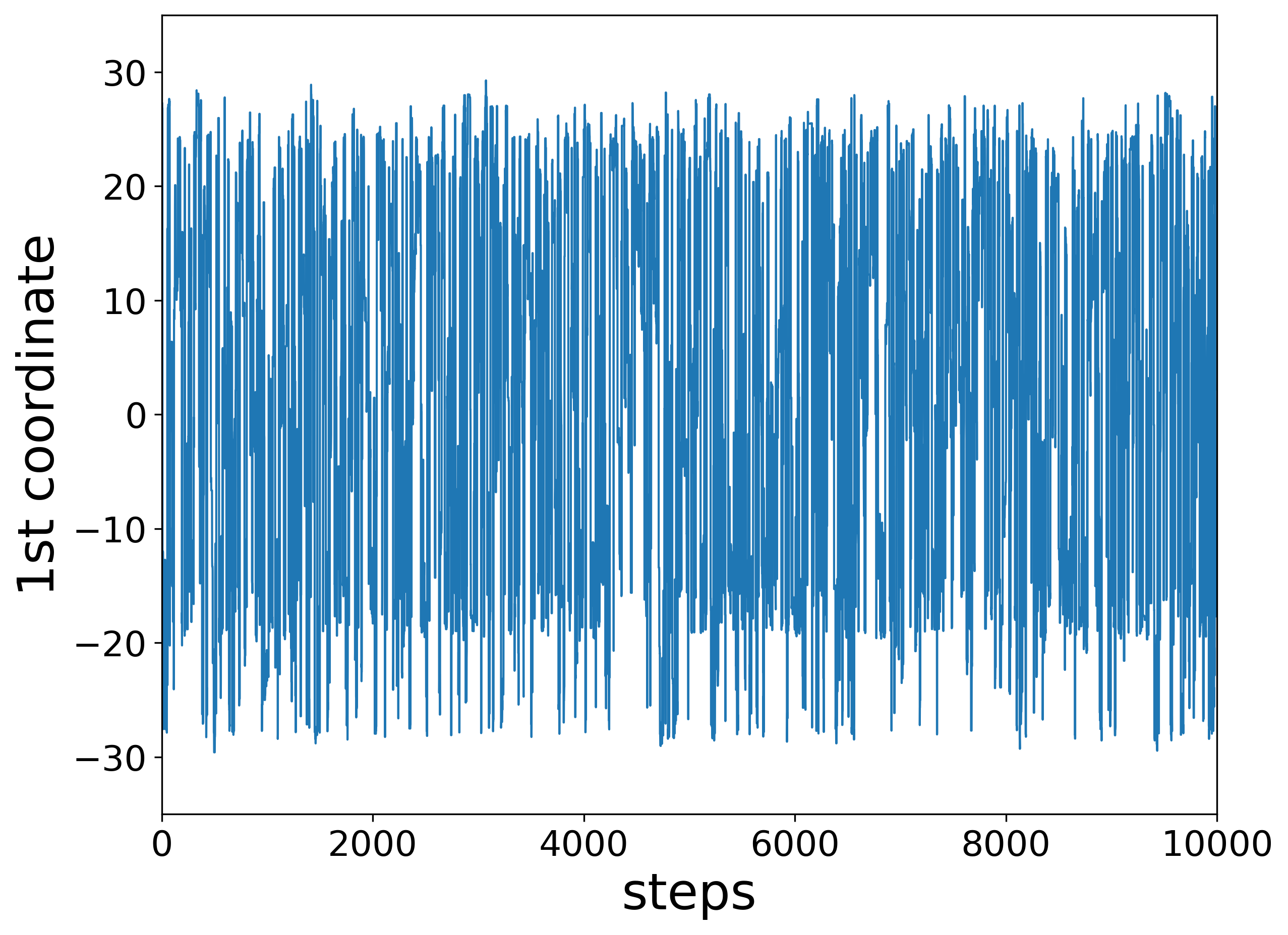}}
\caption{Comparison of RWM and skipping sampler in dimension 2.}\label{fig:2d}
\end{figure*}
\begin{figure*}[!t]
\centering
\subfloat{\includegraphics[width=0.435\textwidth]{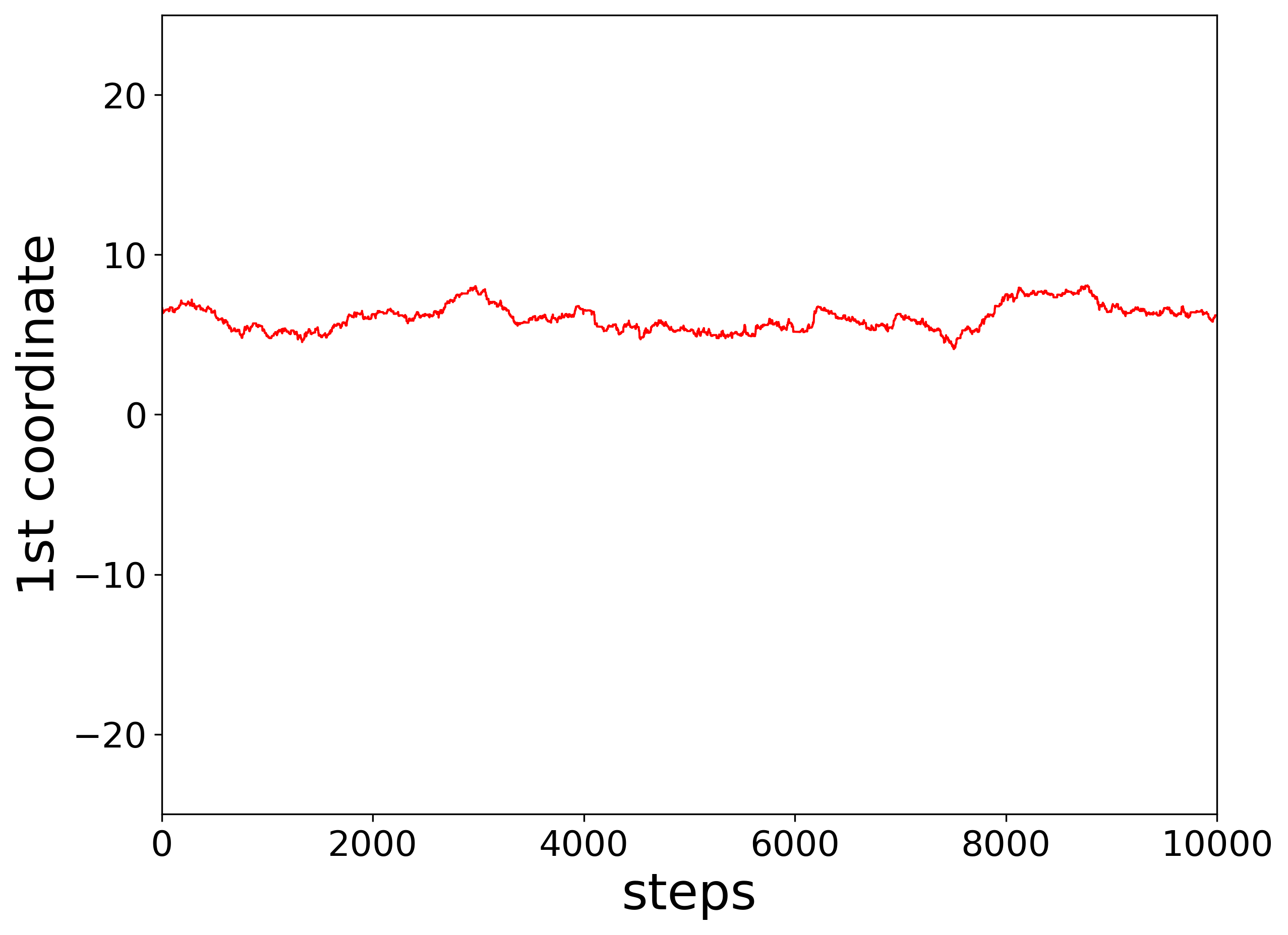}}
\hspace{0.6cm}
\subfloat{\includegraphics[width=0.435\textwidth]{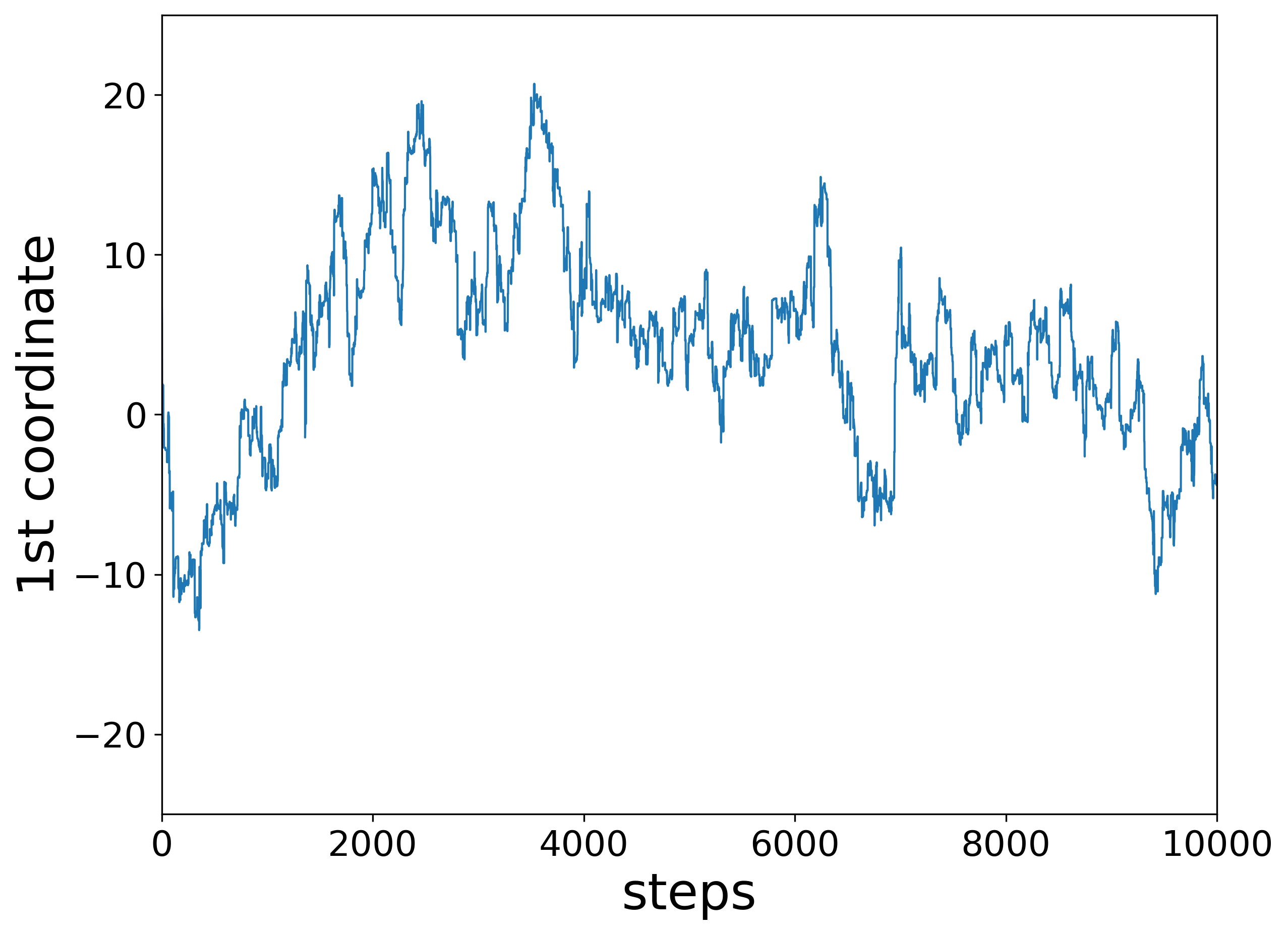}}
\caption{First coordinate trajectory comparison for RWM (left) and skipping sampler (right) in dimension 50.}\label{fig:50d}
\end{figure*}

Figure~\ref{fig:2d} (a)-(b) illustrates that, because of the density's exponential decay, samples from its tail are concentrated around the boundary $\partial A$ of $A$. Figure~\ref{fig:2d} (c)-(d) compares the trajectories of the first coordinate of the chain, showing that while RWM diffuses around $\partial A$, the skipping sampler regularly passes through $A^c$. Indeed, roughly 20\% of the chain's increments were such `skips' through $A^c$, almost half of the accepted proposals. The fact that proposals are often re-used rather than rejected is well illustrated by the acceptance rates, which are $23.7\%$ and $43.3\%$ for RWM and the skipping sampler respectively. Further, $\partial A$ is disconnected. While the `inner' component is not visited by RWM in this sample, the skipping sampler regularly passes through $A^c$ to visit both components, thus exploring $\partial A$ more quickly. Despite $3.45$ 
times more target evaluations required for the skipping sampler, the benefits in this example are clearly worthwhile.

Figure~\ref{fig:50d} shows the evolution of the chain's first coordinate in the case $d=50$. While the boundary of $A$ cannot be easily visualised here, the faster mixing of the skipping sampler is again apparent. 
The successful re-use of proposals by skipping across $A^c$ again constituted approximately $18\%$ of the chain's steps, suggesting superdiffusive exploration. The respective acceptance rates were $22.2\%$ for RWM and $48.1\%$ for the skipping sampler. The benefits of skipping are again seen to be worth the computational cost, since this time the skipping sampler required only $1.44$ times more target evaluations than RWM.

\subsection{Applications to optimisation}
\label{sec:example5}
The challenging problem of finding the global minimum of a non-convex function has attracted much attention and several probabilistic methods and heuristics have been developed, including simulated annealing~\cite{Kirkpatrick1983}, multistart~\cite{Jain1993,Marti2003}, basin-hopping~\cite{Leary2000,Wales1997}, and random search~\cite{Schumer1968}.
In this section we illustrate how the skipping sampler can be used in difficult low-dimensional examples to either bias the choice of initial points of such methods, or as a subroutine, in order to improve exploration.
Below we consider an optimisation problem in $\RR^d$ of the form
\begin{align}\label{eq:opt}
	\min \quad & f(\bm{x}) \quad 
	\text{ s.t. }  
	\, \bm{x} \in D: = \prod_{i=1}^d [l_i,u_i],
\end{align}
and consider as the target density the Bolztmann distribution with temperature $T \geq 0$ and energy function $f$, conditioned on the region $D$, that is
\begin{equation}\label{eq:gibbsdensity}
	\pi(\bm{x}) \propto \exp\left (-f(\bm{x})/T\right )\bm{1}_{\{\bm{x} \in D\}}.
\end{equation} 

\subsubsection{Monotonic skipping sampler}

\begin{figure*}[!t]
\centering
\vspace{-1cm}
\subfloat{\includegraphics[scale=0.265]{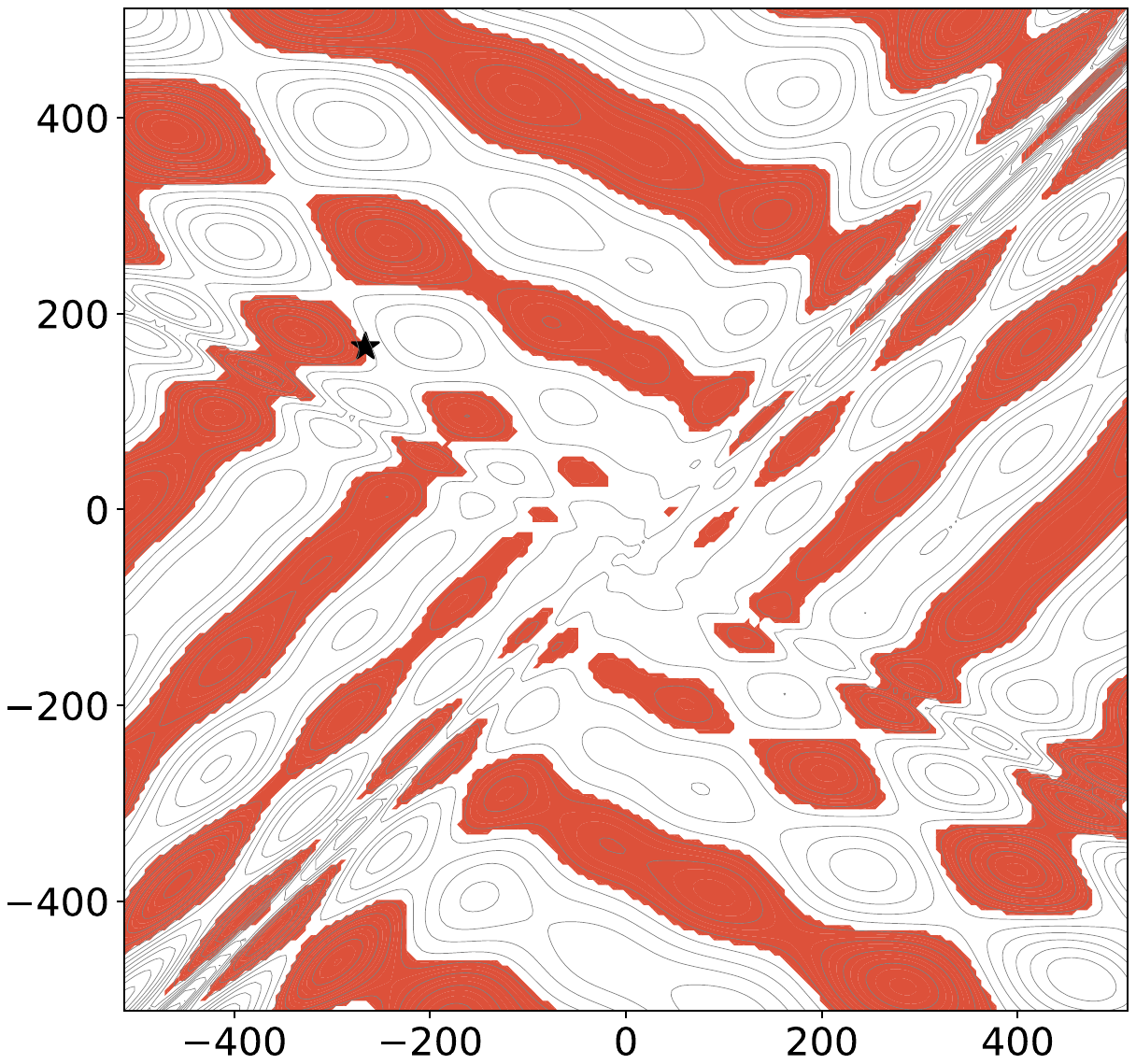}}
\hspace{0.1cm}
\subfloat{\includegraphics[scale=0.265]{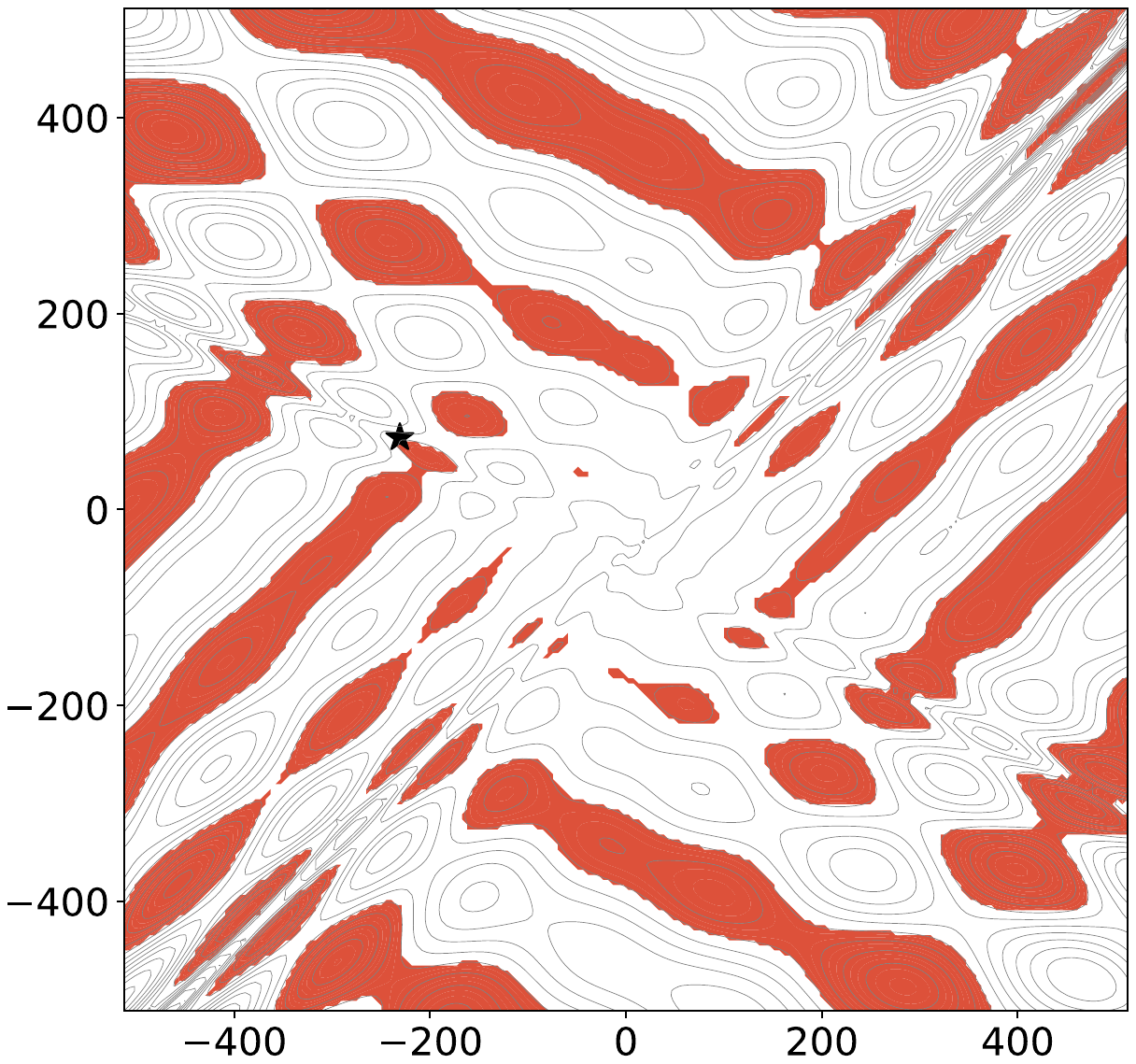}}
\hspace{0.1cm}
\subfloat{\includegraphics[scale=0.265]{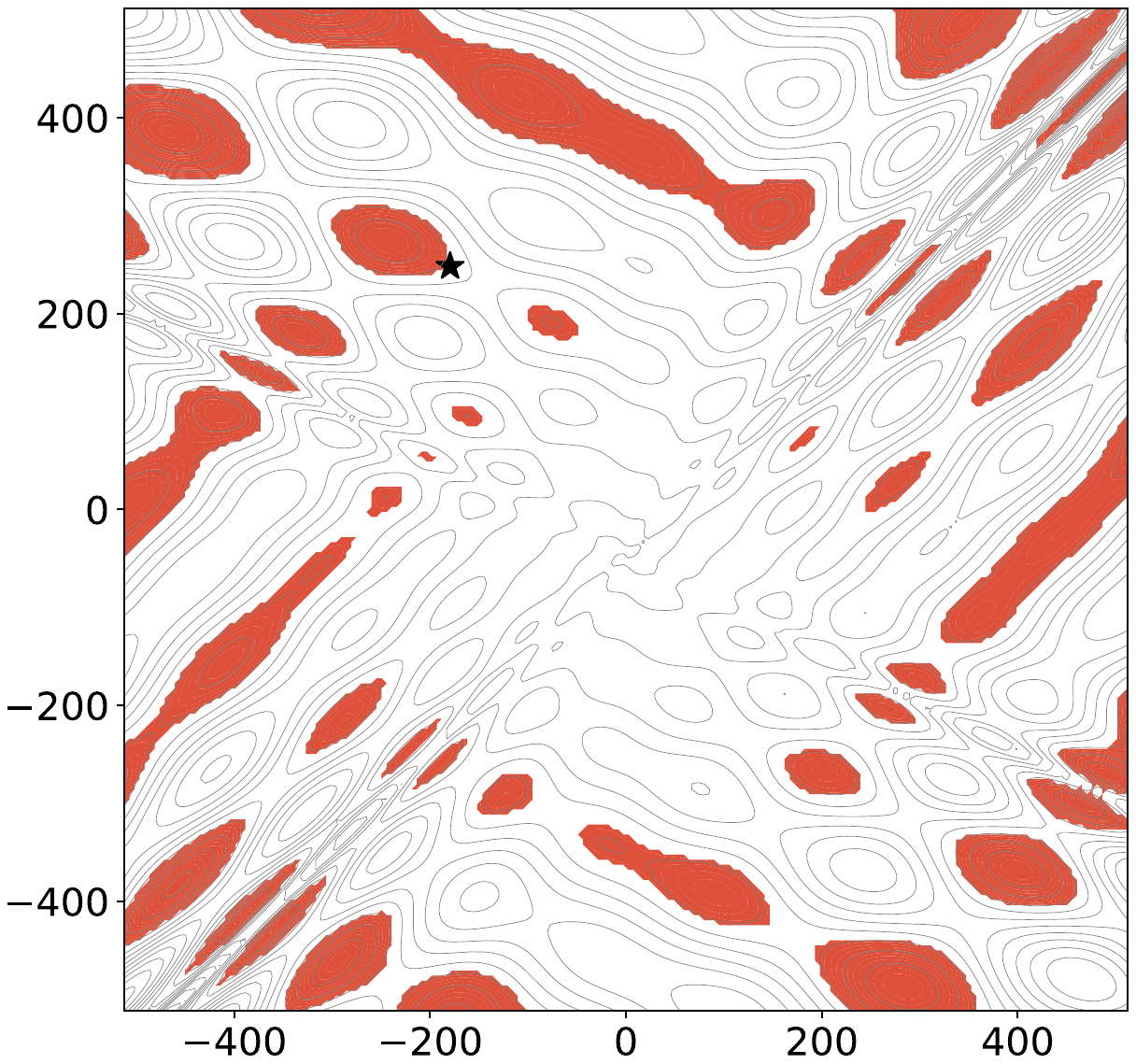}}
\vspace{-2.3cm}\\
\subfloat{\includegraphics[scale=0.265]{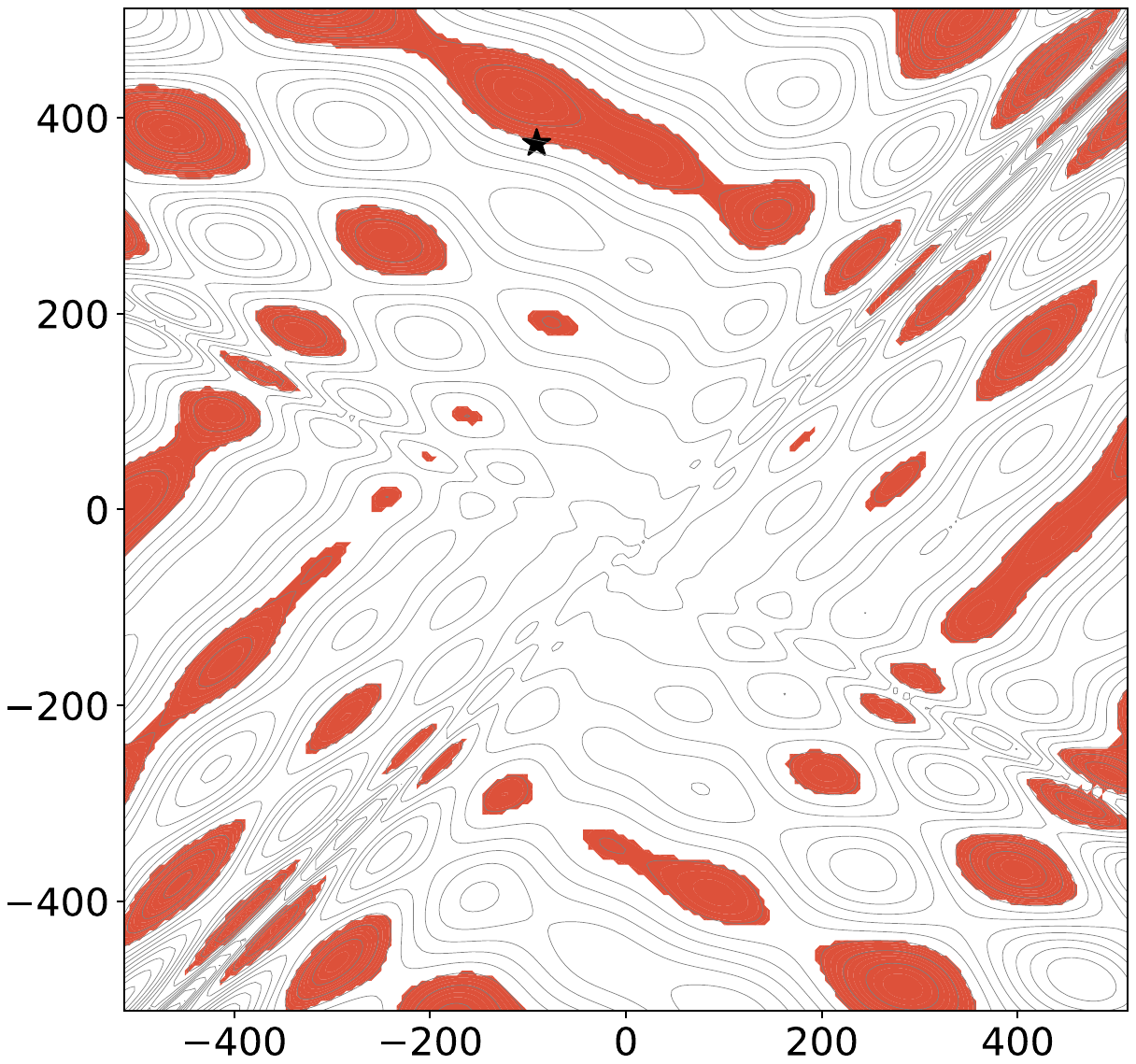}}
\hspace{0.1cm}
\subfloat{\includegraphics[scale=0.265]{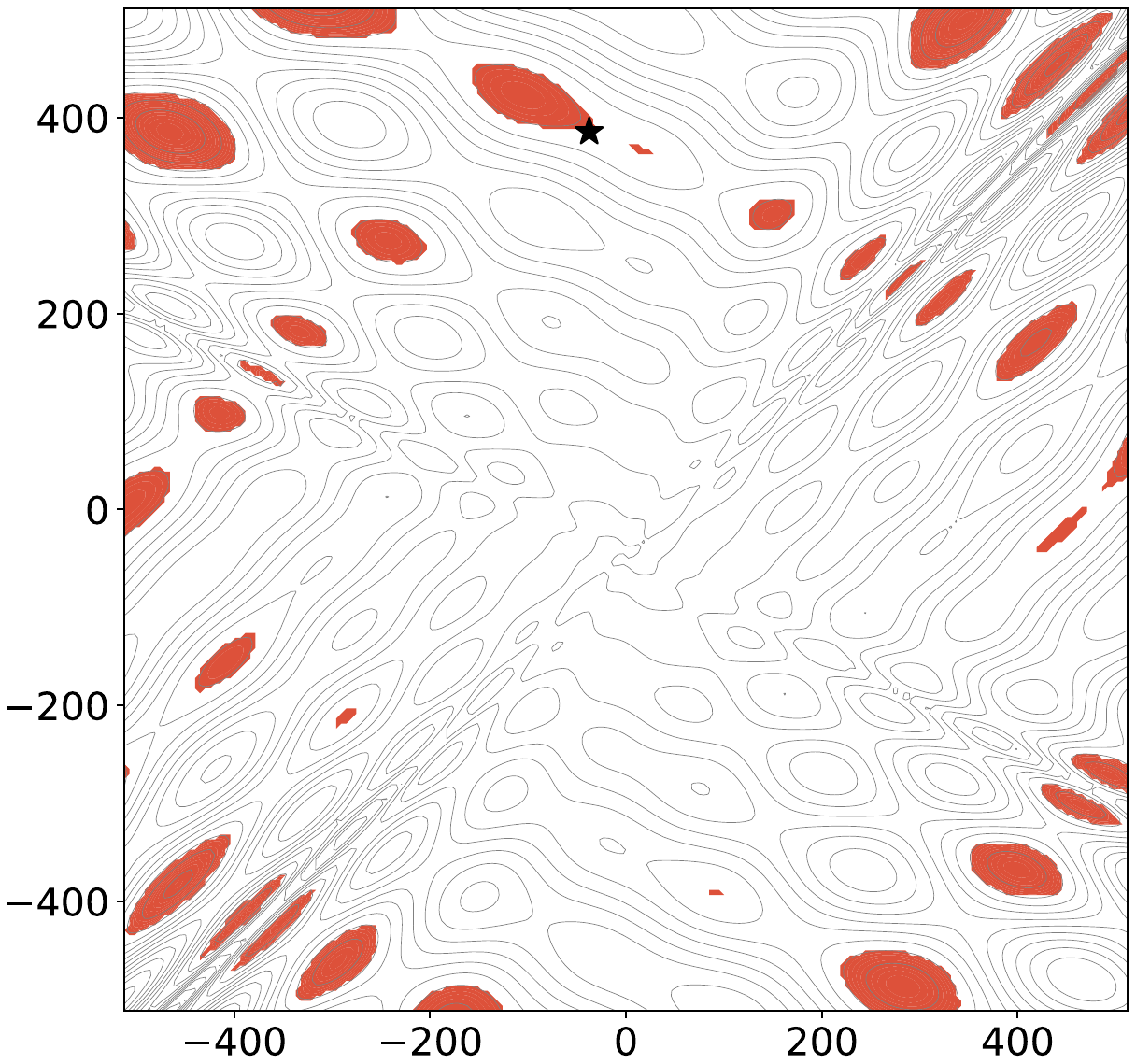}}
\hspace{0.1cm}
\subfloat{\includegraphics[scale=0.265]{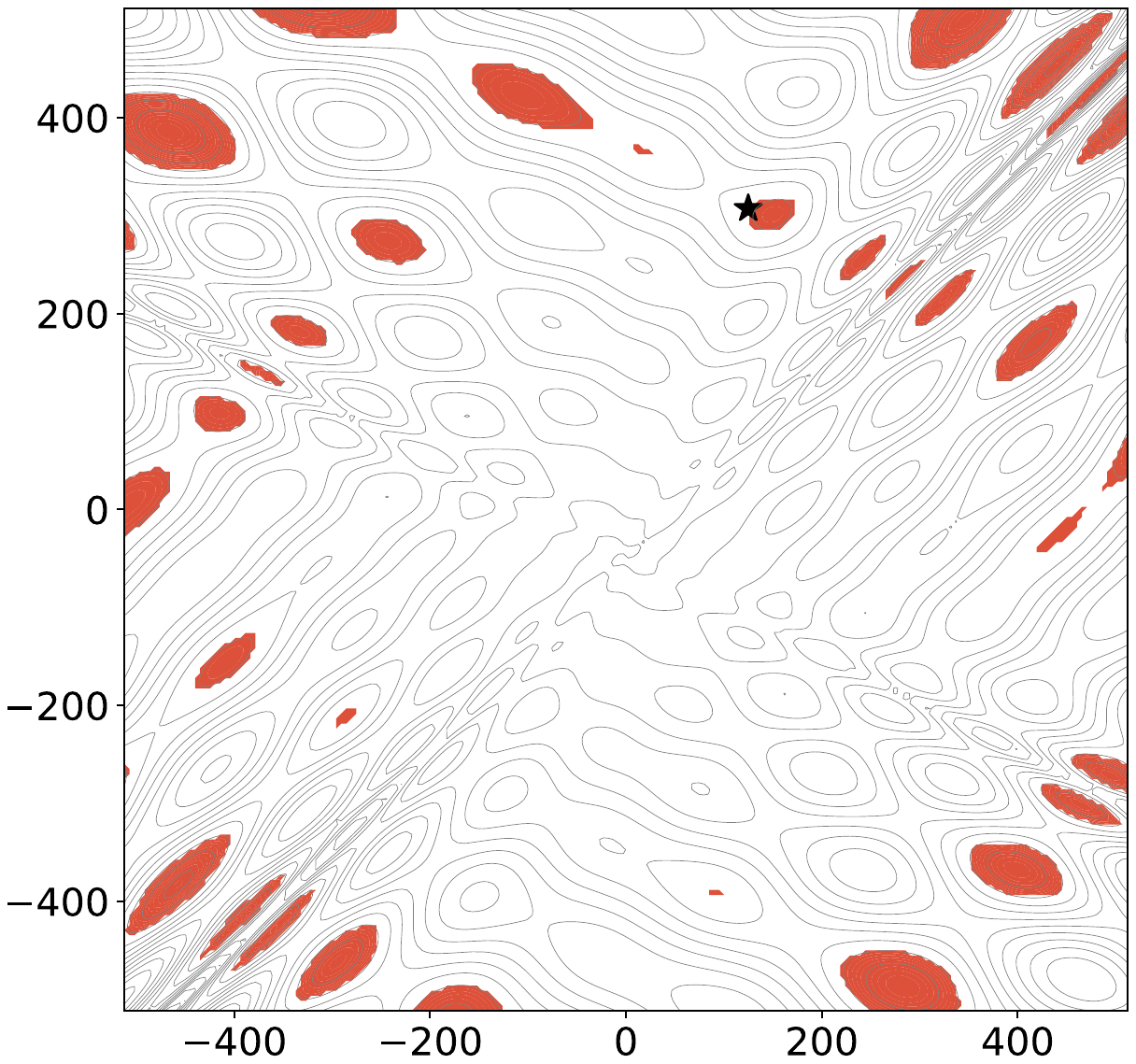}}
\vspace{-1.1cm}
\caption{Given a trajectory $(X_n)_{n=0,1,\dots}$ of the MSS started at $X_0=(-200,180)$, each subfigure displays the point $X_n$ (starred marker) and corresponding sublevel set $S(X_n)$ (in red) for $n=66,67,83$ (first row) and $n=84,107,108$ (second row). There are a total of $54$ skipping moves in this trajectory, which corresponds to $36.0\%$ of the moves. The MSS uses the Gaussian distribution $\mathcal{N}(\bm{0},2\cdot I)$ as proposal density and a deterministic halting index $\mathcal K = 150$.}
\label{fig:trajectoryMSK}
\end{figure*}

While outside the scope of our theoretical analysis, a variation on Algorithm \ref{alg:skip} is one in which the support $A$ is not constant. In particular, defining the level sets $S(X_n) = \{  \bm{x} \in \RR^d ~:~ f(\bm{x}) \leq f(X_n) \}$, a \textit{monotonic skipping sampler (MSS)} may be defined in which the support at the $n$-th step of Algorithm \ref{alg:skip} is $A_n := S(X_n) \cap D$ (setting $A_0 := D$), and the target density $\pi=\pi_n$ is uniform on $A_n$. That is, only downward moves (Markov chain transitions with $f(X_{n+1}) \leq f(X_n)$) are accepted. 
By construction we have $X_n \in A_n$ for each $n \in \mathbb{N}$. Also, since the random subsets $\{A_n\}_{n=1\dots,m}$ are themselves decreasing with $A_{n+1} \subseteq A_{n}$ for every $n$, they contain progressively fewer non-global minima in addition to the global minima of the function $f$.
In common with the skipping sampler where the support $A$ is fixed, the $n$-th step of the MSS requires no information about the sublevel set $S(X_n)$, just the ability to check whether the proposal $Z$ lies in $A_n$. 

To illustrate a trajectory of the MSS, take $f$ to be the so-called \textit{eggholder function} in dimension $d=2$, i.e.
\begin{align*}
	f_\textrm{eggholder}(\bm{x}) \, := \, &-x_1 \sin \left(\sqrt{| x_1-x_2-47| }\right) -(x_2+47) \sin \left(\sqrt{\left| \frac{x_1}{2}+x_2+47\right| }\right),
\end{align*}
an optimisation test function often used in the literature~\cite{Jamil2013}, with $D = [-512,512]^2$. Figure~\ref{fig:trajectoryMSK} shows some snapshots from a trajectory of the MSS, also indicating the progressively shrinking sublevel sets $A_n = S(X_n) \cap D$. In this subsample the state of the chain (starred marker) is seen to jump four times between different connected components of the sublevel sets (in the subfigures for $n=67$, $84$, and $108$), which happens by means of the skipping mechanism.

In Sections \ref{sub:multistart} and \ref{sub:optimisation} we provide numerical examples of performance improvements achieved when the MSS is used as a subroutine in the multistart and basin-hopping optimisation procedures respectively.

\subsubsection{Augmented multistart method}
\label{sub:multistart}

Given a nonconvex optimisation problem of the form~\eqref{eq:opt} with possibly several local minima, a classical strategy to find its global minimum is to restart the local optimisation method of choice at several different points. The multistart method produces the desired number $N$ of initial points by sampling them uniformly at random in $\prod_{i=1}^d [l_i,u_i]$.

Note that in the above setup, $f$ may be set equal to positive infinity outside an arbitrary constraint set. If the set $f^{-1}(\mathbb{R})$ of feasible points has a low volume compared to $D$ then many of the randomly sampled points may lie outside it, making this multistart initialisation procedure inefficient. In this case, recalling the remark on initialisation of Algorithm \ref{alg:skip} from Section \ref{sec:gs}, the MSS is capable of accelerating the search for feasible starting points $\bm{x} \in f^{-1}(\mathbb{R})$.

\FloatBarrier
Equally, sampling starting points uniformly at random may not be helpful if the basin of attraction of the global minimum has low volume. We can mitigate both of these issues by ``improving'' each of the points proposed by the multistart method as follows. Assume $N$ initial points $X^{(1)}_0,\dots,X^{(N)}_0$ have been sampled uniformly at random in $\prod_{i=1}^d [l_i,u_i]$, \textit{which need not be feasible}. For each $i=1,\dots,N$, a Markov chain of length $m$ may be generated using the MSS started at $X^{(i)}_0$, returning $X^{(i)}_m$. Algorithm~\ref{alg:slice} summarises in pseudo-code this MSS-augmented multistart method. By monotonicity, the augmented multistart procedure results in a greater proportion of feasible points, while each initially feasible point is improved.
\begin{table*}[!ht]
\centering
\begin{tabular}{p{4.2cm}|C{3.8cm}|C{3.8cm}|C{3.8cm}}
Metric & Multistart method & Multistart method augmented with $m$ RWM steps & Multistart method augmented with $m$ MSS steps\\ 
\hline 
\hline
Fraction of proposed points in the basin of attraction of the global minimum $x^*$ & 0.004 & 0.008 & 0.657 \\ 
\hline
Euclidean distance from the global minimum $x^*$ & 741.85 (87.94, 1218.84) & 741.85 (87.94, 1218.84) & 0.0 (0.0, 977.869) \\ 
\hline
Gap between the optimum $f(x^*)$ and the objective value of the proposed point & 465.68 (24.30, 848.49) & 465.68 (24.30, 848.49) & 0.0 (0.0, 70.69) \\ 
\hline
Number of function evaluations per run & 36 (21, 84) & 398 (325, 1909)  & 61527 (7805, 133470) \\ 
\hline 
Running time (sec/run) & 0.0021 & 0.0441 & 0.4577 \\ 
\end{tabular}
\caption{Comparison of the quality of the starting points returned by the three variants of the multistart method when solving the optimisation problem~\eqref{eq:eggholder_basic}. The results are averaged over $N=1000$ samples. The RWM and MSS variants both use trajectories of length $m=100$ and the Gaussian proposal density $\mathcal{N}(\bm{0},2\cdot I)$. The halting index for MSS was taken to be deterministic and equal to $\mathcal K=200$. The acceptance probabilities for the RWM are evaluated w.r.t.~the target distribution~\eqref{eq:gibbsdensity} with $T=1.0$. For three of the metrics in the table, we report the median value and, in parenthesis, the 2.5 and 97.5 percentiles.} \label{fig:table1}
\end{table*}
\begin{algorithm}[!ht]
	\SetAlgoLined
	\SetKwInOut{Input}{Input}\SetKwInOut{Output}{Output}
	\Input{The number $N$ of initial points and the desired length $m$ of MSS trajectories}
	\BlankLine
	\DontPrintSemicolon
	Generate $N$ points uniformly at random $X^{(1)}_0,\dots,X^{(N)}_0$ in $\prod_{i=1}^d [l_i,u_i]$; 
	
	\For{$i = 1$ \KwTo $N$}{
    	Starting at $X^{(i)}_0$ generate a trajectory of length $m$ using the MSS
    }
	\KwRet{\textup{the endpoints of the MSS trajectories} $X^{(1)}_m,\dots,X^{(N)}_m$}.\;
	\caption{MSS-augmented multistart method}\label{alg:slice}
\end{algorithm}
To illustrate the potential of the MSS-augmented multistart method, we present an example again using the eggholder function. We first consider the unconstrained optimisation problem 
\begin{align} 
	\min \quad  f_\textrm{eggholder}(\bm{x}) \quad \text{ s.t. } \, \bm{x} \in [-512,512]^2, \label{eq:eggholder_basic}
\end{align}
which has the optimal solution $\bm{x}^*=(512,404.2319)$, attaining the value $f_\textrm{eggholder}(\bm{x}^*)=-959.6407$. 
Averaging over $N=1000$ runs, in Table~\ref{fig:table1} we summarise the ``goodness'' of the $N$ starting points given by the following three methods:
\begin{itemize}
	\item[(i)] multistart method, i.e., initial points uniformly distributed on $[-512,512]^2$;
	\item[(ii)] the initial points obtained in (i) are evolved for $m=100$ steps with a RWM with Gaussian proposals with covariance matrix $2 \cdot I$ and the Boltzmann distribution $\pi$ in~\eqref{eq:gibbsdensity} with $T=1.0$ as target distribution;
	\item[(iii)] the initial points obtained in (i) are evolved for $m=100$ steps with the MSS using a deterministic halting index $\mathcal K = 200$ and the same Gaussian proposal density as in (ii).
\end{itemize}

The MSS augumented multistart method effectively biases the initial points towards the global minimum $x^*$, bringing $65.7\%$ of them in the correct basin of attraction, although at the expense of more function evaluations than the other two methods.

\subsubsection{Skipping sampler as basin-hopping subroutine}
\label{sub:optimisation}

\begin{table*}[!ht]
\centering
\begin{tabular}{p{4.5cm}|C{3.75cm}|C{3.75cm}|C{3.75cm}}
Metric & Basin-hopping (with uniform displacement as subroutine) & Monotonic sequence basin-hopping (with uniform displacement as subroutine) & Basin-hopping with skipping \\ 
\hline 
\hline
Fraction of proposed points in the basin of attraction of the global minimum $x^*$ & 0.022 & 0.017 & 0.544 \\ 
\hline
Euclidean distance from global minimum $x^*$ & 741.85 (87.94, 1248.15) & 752.827 (87.94, 1248.15) & 0.0 (0.0, 41.23) \\ 
\hline
Gap between the optimum $f(x^*)$ and the objective value of the proposed point & 419.21 (24.30, 835.76) & 402.03 (24.30, 835.85) & 0.0 (0.0, 2.72) \\
\hline
Number of function evaluations & 2004 (1776, 3432) & 348 (324, 5085) & 20370 (19957, 23702) \\ 
\hline 
Running time (sec/run) & 0.123 & 0.065 & 0.333 \\ 
\end{tabular}
\caption{Performance comparison of the three basin-hopping variants using different subroutines averaged over $N=1000$ samples for trajectories of length $m=100$. The underlying proposal used for the MSS subroutine is a standard Gaussian distribution $\mathcal{N}(\bm{0},I)$ and the uniform displacement of the other two methods is scaled to have the same standard deviation. The halting index for MSS was taken to be deterministic and equal to $\mathcal K=200$. The acceptance probabilities for the basin-hopping methods are evaluated w.r.t.~the target distribution~\eqref{eq:gibbsdensity} with $T=1.0$. For three of the metrics in the table, we report the median value and, in parenthesis, the 2.5 and 97.5 percentiles.} \label{fig:table_opt}
\end{table*}

Besides improving the multistart method, the skipping sampler can also be used to improve stochastic techniques for non-convex optimisation, in particular the so-called \textit{basin-hopping method}. In this subsection we explore this novel idea, although the implementation details and a systematic comparison with other global optimisation routines are left for future work. 

Basin-hopping is a global optimisation technique proposed in~\cite{Wales1997}, which at each stage combines a random perturbation of the current point, local optimisation, and an acceptance/rejection step. The random perturbation consists of i.i.d.~uniform simultaneous perturbations in each of the coordinates, that is, a random walk step.
The stopping criterion for this iterative procedure is often a maximum number of function evaluations, or when no improvement is observed for a certain number of consecutive iterations. 

The random walk step may be replaced by a step from the MSS. That is, at step $n$, given the current point $X_{n-1}$ we first sample a new point $Y_{n}$ from the sublevel set $S(X_{n-1}) \cap D$ using MSS and then perform a local optimisation procedure starting from $Y_{n}$ to obtain a new point $X_{n}$. This idea is summarised in the pseudo-code presented in Algorithm~\ref{alg:optimize}.

\begin{algorithm}[!ht]
	\SetAlgoLined
	\SetKwInOut{Input}{Input}\SetKwInOut{Output}{Output}
	\Input{An initial point $X_0 \in \prod_{i=1}^d [l_i,u_i]$, not necessarily feasible}
	\BlankLine
	\DontPrintSemicolon
	Set $n=1$;\\
	\While{basin-hopping stopping criterion is not satisfied}{
	    Perform a single step of the MSS started at $X_{n-1}$, obtaining a new point $Y_{n}$;\\
		Perform a local optimisation step started at $Y_{n}$ to obtain a new point $X_n$;\\	
    	Increment the index $n$ by one;  
    }
	\KwRet{\textup{The last point} $X_n$}.\;
	\caption{Basin-hopping with skipping}\label{alg:optimize}
\end{algorithm}

The MSS variant of the basin-hopping method is related to the monotonic sequence basin-hopping (MSBH) proposed in~\cite{Leary2000}, which also accepts only new points in $S(Y_{n-1}) \cap D$. However MSBH uses only local uniform perturbations and thus faces the same exploration challenges as RWM when $S(Y_{n-1}) \cap D$ is disconnected. 

In Table~\ref{fig:table_opt}, we compare the performance of basin-hopping and of the MSBH with the proposed basin-hopping with skipping.
The MSS subroutine leads to $54.4\%$ of the initial (uniformly distributed) points converging to the basin of attraction of the global minimum $x^*$. This sharp improvement with respect to basin-hopping (which has a corresponding success rate of only $2.2\%$) only requires ten times more function evaluations.
In~\cite{Goodridge2021} we present additional performance metrics for the basin-hopping method with skipping and more extensive numerical results over a large collection of test functions.

\section{Proofs}\label{sec:proof}

\subsection{Proof of Theorem~\ref{thm:ergodic}}
For each $x \in \textrm{supp}(\pi)$ let $\PP_x$ be a probability measure carrying all  random variables used in Algorithm \ref{alg:skip}, such that $X_0 = x$ almost surely under $\PP_x$. Denote respectively by $\{Y_m\}_{m \geq 1}$ and $\{X_n\}_{n \geq 1}$ the proposals generated by the $\MH(\pi,q)$ algorithm and the Markov chain returned by the algorithm. Writing $\cA_n:=\bigcap_{i=1}^n \{X_i=Y_i \}$ for the event that the first $n$ proposals of $\MH(\pi,q)$ are all accepted, we have  

\begin{lemma}\label{lem:nonz}
If the chain $\mathrm{MH}(\pi,q)$ restricted to $\mathrm{supp}(\pi)$ is $\pi$-irreducible then $\PP_x\left(\cA_{m}\right)>0$ for all $x \in \textrm{supp}(\pi)$ and all $m \geq 1$. 
\end{lemma}
\begin{proof}
Fixing $x \in \textrm{supp}(\pi)$ and supposing otherwise for a contradiction, let $n$ be the smallest integer such that $\PP_x\left(\cA_{n}\right)=0$. Clearly $n \geq 2$, since otherwise, $\PP_x-$almost surely we have $X_k=X_0$ for all $k \geq 1$, contradicting the assumption of $\pi$-irreducibility. Therefore $\PP_x\left(\cA_{n-1}\right)>0$ and we may write $p$ for the density of $X_{n-1}$ conditional on the event $\cA_{n-1}$. Then by the Markov property we have
\begin{align*}
	0 & \quad = \quad \PP_x\left(\cA_{n-1}\right) \PP_x\left(\cA_{n}|\cA_{n-1}\right) \quad = \quad \PP_x\left(\cA_{n-1}\right) \int_{\mathrm{supp}(p)} p(y) \, \PP_y(\cA_1) \, dy ,
\end{align*}
so that $\PP_y(\cA_1)=0$ for some $y \in \mathrm{supp}(p)$. Arguing as above, this contradicts the assumption of $\pi$-irreducibility.
\end{proof}

Denote the Markov kernels of the chains generated by $\MH(\pi,q)$ and $\MH(\pi,q_\cK)$  by $P$ and $P_\cK$ respectively. Also let $\{X'_n\}_{n \geq 1}$ be the jump chain associated with $X$ (that is, the subsequence of $\{X_n\}_{n \geq 1}$ given by excluding all $X_m$ which satisfy $X_m = X_{m-1}$). 

\begin{lemma}\label{lem:irredaux}
For all $x\in \K=\mathrm{supp}(\pi)$, $n\in\NN$ and all $B \subset \K$ the following inequality holds:
\begin{align}
	P_\cK^n(x,B) & \quad \geq \quad \PP_x\left (\{X_n\in B\} \cap \cA_n\right)  \nonumber \\
	& \quad = \quad \PP_x\left (X_n\in B ~|~ \cA_n\right) \PP_x\left(\cA_n\right) \nonumber \\
	& \quad = \quad \PP_x\left (X'_n\in B ~|~ \cA_n \right)\PP_x\left(\cA_n\right). \label{eq:jump}
\end{align}
\end{lemma}

\begin{proof} Note first that the last equality in \eqref{eq:jump} follows by definition of the jump chain. 
We will prove the inequality in \eqref{eq:jump} by induction on $n$. Since $\mathrm{supp}(\pi)=\K$, Proposition~\ref{prop:m} (ii) gives 
\begin{align*}
	P_{\cK}(x,B) & \quad \geq \quad \int_B\alpha(x,z)q_\mathcal{K}(x,z)dz \\& \quad \geq \quad \int_B\alpha(x,z)q(z-x)dz \\
	& \quad = \quad \PP_x\left (\{X_1\in B\} \cap \cA_1\right).
\end{align*}
Assume now the statement holds for some $n\in\NN$ and let us prove it for $n+1$. We argue using the induction hypothesis and Proposition~\ref{prop:m} (ii) again:
\medmuskip=1mu
\thinmuskip=1mu
\thickmuskip=1mu
\begin{align*}
P_{ \cK}^{n+1}(x,B) \quad
	& \quad = \quad \int_{\K}P_{\cK}^n(z,B)P_{ \cK}(x,dz)\\
	& \quad geq \quad \int_{\K}P_{ \cK}^n(z,B)\alpha(x,z)q_\mathcal{K}(x,z)\,dz\\
	& \quad \geq \quad \int_{\K}P_{ \cK}^n(z,B)\alpha(x,z)q(z-x)\,dz\\
	& \quad \geq \quad \int_{\K}\PP_z\left (\{X_n\in B\} \cap \cA_n\right)\alpha(x,z)q(z-x)\,dz\\
	&\quad = \quad \PP_x\left (\{X_{n+1}\in B\} \cap \cA_{n+1}\right)\,. \qedhere
\end{align*}
\thinmuskip=3mu
\medmuskip=4mu
\thickmuskip=5mu
\end{proof}

\begin{proof}[Proof of Theorem \ref{thm:ergodic}]
Take $B\subseteq \K=\mathrm{supp}(\pi)$ such that $\pi(B)>0$, $x\in\K$ and let $\{X_n\}_{n \geq 1}$ be $\MH(\pi,q)$ started at $X_0=x$. Since $\MH(\pi,q)$ is $\pi$-irreducible there exists an integer $n\in\NN$ such that $\mathbb{P}_x\left(X_{n}\in B\right)>0$. Let $S_n$ be the number of rejections which occur in the generation of $\{X_m\}_{1 \leq m \leq n}$. Then 
$$
	0 \quad<\quad \mathbb{P}_x\left(X_{n}\in B\right) 
	\quad=\quad
	 \sum_{i=0}^n \mathbb{P}_x\left(X_{n}\in B, S_n=i\right).
$$
For some $j \in \{1,\ldots,n\}$ we therefore have 
$$
   \mathbb{P}_x\left(X_{n}\in B, S_n=j\right)>0.
$$
Consequently
$$
    \mathbb{P}_x\left(X_{n-j}'\in B|\mathcal{A}_{n-j}\right)>0,
$$
so that
\begin{align*}
	P_\cK^{n-j}(x,B) \geq \mathbb{P}_x\left(X'_{n-j}\in B ~|~ \cA_{n-j} \right) \PP_x\left(\cA_{n-j}\right)
	> 0\,,
\end{align*}
where we used the above together with Lemma~\ref{lem:irredaux} and Lemma \ref{lem:nonz}.
The skipping chain $\MH(\pi,q_\mathcal{K})$ is therefore $\pi$-irreducible, and thus is Harris recurrent by \cite[Corollary~2]{tierney}. Furthermore, \cite[Theorem~10.0.1]{tweedie} yields that $\pi$ is its unique invariant probability measure. Finally, the SLLN holds for all $\pi$-integrable functions by Harris recurrence and \cite[Theorem~17.1.7]{tweedie}. 
\end{proof}

\subsection{Proof of Theorem~\ref{thm:CLT}}
To prove Theorem~\ref{thm:CLT}, we will make use of the following lemma, whose proof is omitted.
\begin{lemma}[Integration with respect to a symmetric joint density] 
\label{lem:integration}
Consider a symmetric density $\Delta: \RR^d\times \RR^d \to [0,+\infty)$ and a subset $B \subseteq \RR^d$. For every $f \in L^2(\Delta)$ the following identity holds:
\medmuskip=-1mu
\thinmuskip=-1mu
\thickmuskip=-1mu
$$
	\int_{B} \int_{B} \frac{f(x)^2+f(y)^2}{2} \Delta(x,y) dy dx \quad
	= \quad \int_{B}  f(x)^2  \left (\int_{B} \Delta(x,y) dy \right ) dx\,.
$$
\thinmuskip=3mu
\medmuskip=4mu
\thickmuskip=5mu
\end{lemma}

\begin{proof}[Proof of Theorem~\ref{thm:CLT}]
(i) For any $f \in L^2(\pi)$ the desired inequality $\langle P_\mathcal{K} f,f\rangle\leq \langle P f,f\rangle$ can be written more explicitly as 
\medmuskip=0mu
\thinmuskip=0mu
\thickmuskip=0mu
\begin{align*}
	& \int_{\RR^d}f(x) \Bigg ( \left(\int_{\RR^d}f(y)\alpha(x,y) ( q_\mathcal{K}(x,y) - q(y-x) )dy\right) + f(x)(r_\mathcal{K}(x) - r(x))\Bigg )\pi(x)dx \quad\leq\quad 0,
\end{align*}
\thinmuskip=3mu
\medmuskip=4mu
\thickmuskip=5mu
where we respectively denote by $r(x)$ and $r_\mathcal{K}(x)$ the rejection probabilities starting at point $x$ of $\MH(\rho,q)$ and $\MH(\rho,q_\mathcal{K})$, i.e., $r(x):=1-\int_{\RR^d}\alpha(x,y)q(y-x)dy$ and analogously for $r_\mathcal{K}(x)$.
The above inequality holds provided that we establish the following one:
\medmuskip=0mu
\thinmuskip=0mu
\thickmuskip=0mu
\begin{multline} \label{eq:l2ineq_step}
	\int_{\RR^d}f(x)\left(\int_{\RR^d}f(y)\alpha(x,y)\left(q_\mathcal{K}(x,y)-q(y-x)\right)dy\right)\pi(x)dx\quad \leq\quad \int_{\RR^d}f^2(x)(r(x)-r_\mathcal{K}(x))\pi(x)dx.
\end{multline}

Then considering the LHS of~\eqref{eq:l2ineq_step} and Proposition~\ref{prop:m} (ii) we have:
\medmuskip=0mu
\thinmuskip=0mu
\thickmuskip=0mu
\begin{align*}
&\int_{\RR^d}f(x)\left(\int_{\RR^d}f(y)\alpha(x,y)\left(q_\mathcal{K}(x,y)-q(y-x)\right)dy\right)\pi(x)dx
\\
& \quad = \quad \int_{\K}\int_{\K}f(y)f(x)\alpha(x,y)\pi(x)\left(q_\mathcal{K}(x,y)-q(y-x)\right)dydx
\\
&\quad \leq \quad \int_{\K}\int_{\K}\frac{f^2(y)+f^2(x)}{2}\alpha(x,y)\pi(x)\left(q_\mathcal{K}(x,y)-q(y-x)\right)dydx\\
&\quad \stackrel{(\star)}{=} \quad \int_{\K}\int_{\K}  f(x)^2 \alpha(x,y)\pi(x)\left(q_\mathcal{K}(x,y)-q(y-x)\right)dydx\\
&\quad = \quad \int_{\K} f(x)^2 \left(\int_{\K}\alpha(x,y)\left(q_\mathcal{K}(x,y)-q(y-x)\right)dy\right)\pi(x)dx\\
&\quad = \quad \int_{\RR^d}f^2(x)(r(x)-r_\mathcal{K}(x))\pi(x)dx\,.
\end{align*}
\thinmuskip=3mu
\medmuskip=4mu
\thickmuskip=5mu
In this derivation we used (in order) the fact that $\alpha(x,y)=0$ for $y\in \Kc$ by definition of $\alpha$ and $\pi$ and the classical GM-QM inequality $2f(x)f(y)\leq f(x)^2+f(y)^2$. Furthermore, equality $(\star)$ holds thanks to Lemma~\ref{lem:integration} by taking $\Delta(x,y)= \alpha(x,y) (q_\mathcal{K}(x,y)-q(y-x) ) \pi(x)$ and $B = \K$. 
The property that $\Delta(x,y)=\Delta(y,x)$ for every $x,y \in \K$ readily follows by combining the following two identities that hold for every $x,y\in \K$:
\begin{align*}
	& \alpha(x,y)\pi(x)=\min(\pi(x),\pi(y))=\alpha(y,x)\pi(y), \quad \text{and} \quad q_\mathcal{K}(x,y)-q(y-x)=q_\mathcal{K}(y,x)-q(x-y).
\end{align*}
The first identity is an immediate consequence of the definition~\eqref{eq:ap2} of $\alpha$, while the second one follows from Assumption~\ref{ass:qSRW} and Proposition~\ref{prop:m} (i).\\

(ii) By (i) we have $\langle (I-\gP)f,f\rangle\geq \langle (I-P)f,f\rangle$ for all $f\in L^2(\pi)$. The proof follows by $\lambda_\mathcal{K}=\inf_{f\in \mathcal{M}}\langle (I-\gP)f,f\rangle\geq \inf_{f\in \mathcal{M}}\langle (I-P)f,f\rangle=\lambda$ where $\mathcal{M}=\{f\in L^2(\pi) ~:~ \pi(f^2)=1, \, \pi(f)=0\}$.\\

(iii) This follows by (i) and \cite[Theorem 6]{mira2009covariance}.
\end{proof}

\textbf{Acknowledgements:}
JM was supported by EPSRC grant EP/P002625/1, JV by EPSRC grants EP/N001974/1 and EP/R022100/1, and AZ by NWO Rubicon grant 680.50.1529. The authors wish to thank Wilfrid Kendall, Krzysztof {\L}atuszy{\'n}ski and Andrew Duncan for useful discussions, and the associate editor and two anonymous referees for their comments which helped improve the manuscript. The authors would like to thank the Isaac Newton Institute for Mathematical Sciences for support and hospitality during the programme ``Mathematics of Energy Systems'' when work on this paper was undertaken. This work was supported by EPSRC grant number EP/R014604/1.


\end{document}